\documentclass[onefignum,onetabnum]{siamonline181217}

\usepackage{graphicx}
\usepackage{amsfonts}
\usepackage{amsmath,amssymb}
\usepackage[ruled,linesnumbered,vlined,algo2e,resetcount]{algorithm2e}
\usepackage{commath}
\usepackage{xcolor}
\usepackage{caption}
\usepackage{float}
\usepackage{setspace}
\usepackage{enumerate}
\usepackage{hyperref}
\usepackage{bm}
\usepackage{multirow}
\usepackage{cite}
\usepackage{microtype}
\usepackage{verbatim}
\usepackage{array}
\usepackage{graphicx}
\usepackage{subcaption}

\newcolumntype{C}[1]{>{\centering\let\newline\\\arraybackslash\hspace{0pt}}m{#1}}

\newcommand\restr[2]{{
  \left.\kern-\nulldelimiterspace 
  #1 
  \vphantom{\big|} 
  \right|_{#2} 
  }}

\SetKwInput{KwInput}{Input}                
\SetKwInput{KwOutput}{Output}              

\graphicspath{ {/images/} }

\newsiamremark{remark}{Remark}
\newsiamremark{hypothesis}{Hypothesis}
\crefname{hypothesis}{Hypothesis}{Hypotheses}
\newsiamthm{claim}{Claim}

\headers{Parallelizable global QC parameterization of multiply-connected surfaces}{Zhipeng Zhu, Gary P. T. Choi, and Lok Ming Lui}

\title{Parallelizable global quasi-conformal parameterization of multiply-connected surfaces via partial welding\thanks{Submitted to the editors DATE.
\funding{This work was supported in part by the National Science Foundation under Grant No.~DMS-2002103 (to Gary P. T. Choi), and HKRGC GRF under project ID 14305919 (to Lok Ming Lui).}}}

\author{Zhipeng Zhu\thanks{
              Department of Mathematics, The Chinese University of Hong Kong
              (\email{zpzhu@math.cuhk.edu.hk}).}
           \and
           Gary P. T. Choi\thanks{
              Department of Mathematics, Massachusetts Institute of Technology
              (\email{ptchoi@mit.edu}).}
           \and
           Lok Ming Lui\thanks{
              Department of Mathematics, The Chinese University of Hong Kong
              (\email{lmlui@math.cuhk.edu.hk}).}
}

\usepackage{amsopn}

\ifpdf
\hypersetup{
  pdftitle={Parallelizable global quasi-conformal parameterization of multiply-connected surfaces via partial welding},
  pdfauthor={Zhipeng Zhu, Gary P. T. Choi, Lok Ming Lui}
}
\fi

\begin{document}

\maketitle

\begin{abstract}
Conformal and quasi-conformal mappings have widespread applications in imaging science, computer vision and computer graphics, such as surface registration, segmentation, remeshing, and texture map compression. While various conformal and quasi-conformal parameterization methods for simply-connected surfaces have been proposed, efficient parameterization algorithms for multiply-connected surfaces are less explored. In this paper, we propose a novel parallelizable algorithm for computing the global conformal and quasi-conformal parameterization of multiply-connected surfaces onto a 2D circular domain using variants of the partial welding method and the Koebe's iteration. The main idea is to first partition a multiply-connected surface into several subdomains and compute the free-boundary conformal or quasi-conformal parameterizations of them respectively, and then apply a variant of the partial welding algorithm to reconstruct the global mapping. We apply the Koebe's iteration together with the geodesic algorithm to the boundary points and welding paths before and after the global welding to transform all the boundaries to circles conformally. After getting all the updated boundary conditions, we obtain the global parameterization of the multiply-connected surface by solving the Laplace equation for each subdomain. Using this divide-and-conquer approach, the global conformal and quasi-conformal parameterization of surfaces can be efficiently computed. Experimental results are presented to demonstrate the effectiveness of our proposed algorithm. More broadly, the proposed shift in perspective from solving a global quasi-conformal mapping problem to solving multiple local mapping problems paves a new way for computational quasi-conformal geometry.
\end{abstract}



\section{Introduction}
\label{intro}
In modern applied mathematics and computer science, three-dimensional (3D) surfaces play an important role in many fields, such as brain mapping in medical imaging, 3D model reconstruction in computer graphics, and 3D object detection and classification in computer vision. One important technique for processing 3D models is surface parameterization, which refers to the process of mapping a 3D surface to a two-dimensional (2D) domain based on certain criteria. With the aid of surface parameterization, one can work on the 2D domain instead of on the original 3D surface. For example, to solve a partial differential equation (PDE) on a complicated 3D domain, one can map the domain to a 2D parameter domain and then solve the PDE on it instead. Moreover, with the advancement of 3D scanning and rendering technologies, 3D surfaces with super large size and high resolution can be easily obtained nowadays. Therefore, fast and accurate algorithms for the parameterization of large meshes arise in need. 

Among all the surface parameterization methods, conformal parameterizations are a very special class. Conformality preserves the angular structure at the infinitesimal level, and thus preserves the local geometry. This property is advantageous in many tasks that rely on the preservation of local geometry, such as 3D surface remeshing and image registration. Quasi-conformal (QC) maps are a generalization of conformal maps associated with a complex-valued function defined at each point of the source domain called the Beltrami coefficient. Unlike conformal maps, quasi-conformal maps do not preserve local geometry in general. In particular, the Beltrami coefficient defined at each point of the source domain determines the angular distortion at the infinitesimal level at these points. Also, the bijectivity of quasi-conformal maps can be ensured by enforcing the sup-norm of the Beltrami coefficient to be strictly less than 1. Since conformality is a very strict condition that cannot be ensured in many situations with the presence of other constraints, quasi-conformal maps are often utilized. For instance, in image and surface registration, quasi-conformal maps can be used for achieving a balance between the local geometric distortion and the mismatch in prescribed landmark or intensity information of the registered images and surfaces. In recent years, various algorithms have been proposed for computing conformal and quasi-conformal maps. However, most of them are not designed for large meshes, especially those with more complicated topology such as multiply-connected meshes. 

\begin{figure}[t]
\centering
  \includegraphics[width=\linewidth]{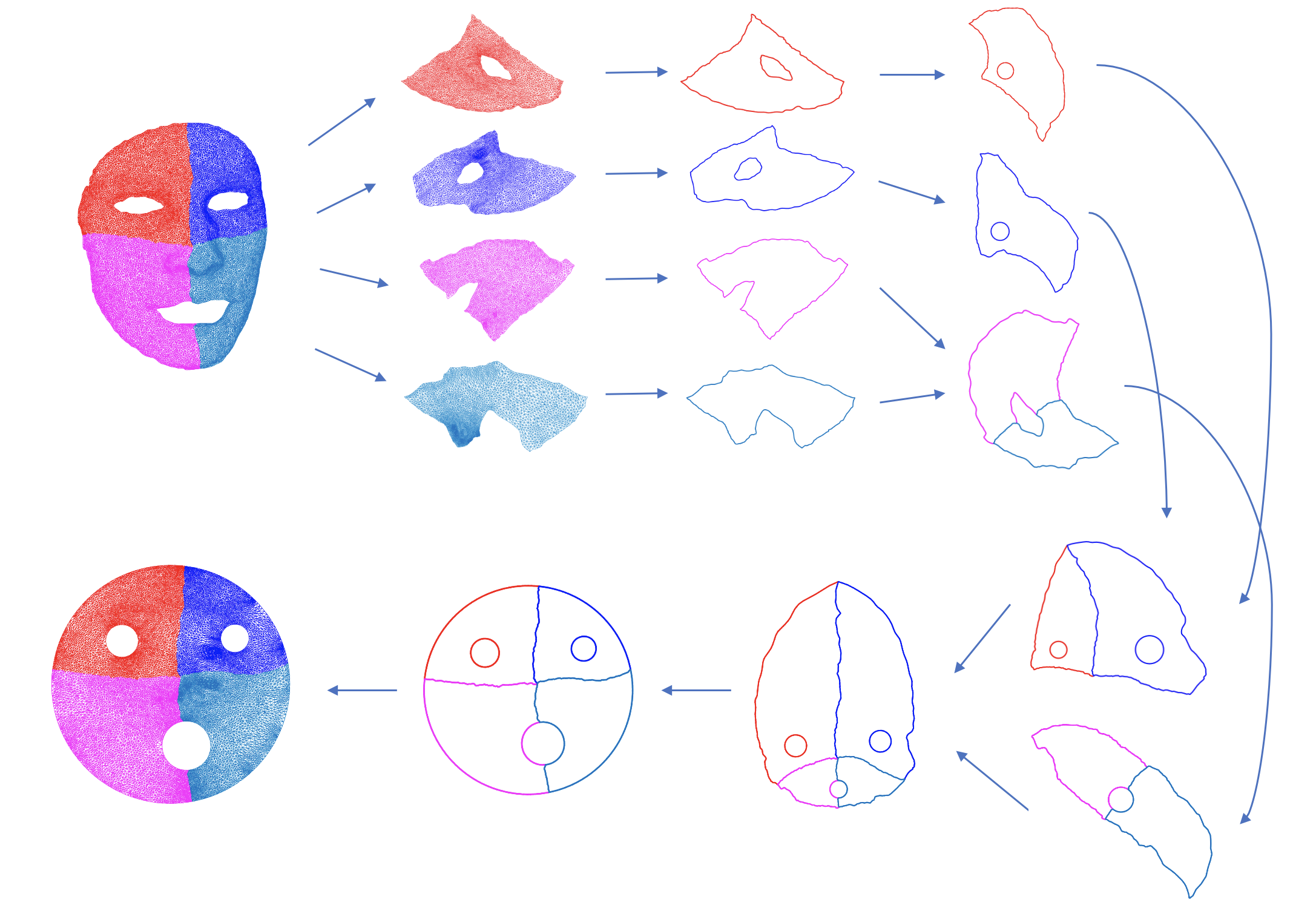}
\caption{An illustration of our proposed method for the global quasi-conformal parameterization of multiply-connected surfaces. Given a multiply-connected surface partitioned into several subdomains, we first compute the free-boundary quasi-conformal parameterization for each of them in parallel. Then, we apply our proposed variants of the partial welding method~\cite{choi2020parallelizable} and the Koebe's iteration~\cite{koebe1910konforme} to find a global conformal mapping of the boundaries of the submeshes to a circular domain with circular holes. Finally, we obtain the global parameterization of the entire surface by solving the Laplace equation on each subdomain in a parallel manner. Note that most of the steps for different subdomains are independent of the other subdomains and hence the method is highly parallelizable. } 
\label{fig:overview}
\end{figure} 

In this paper, we propose a novel parallelizable method for the computation of quasi-conformal parameterization of multiply-connected surfaces onto a 2D circular domain, which refers to a connected domain whose complements are several circular disks. As a special case of quasi-conformal maps, conformal maps can also be efficiently computed by our method. Fig.~\ref{fig:overview} gives an overview of our proposed method. Given a multiply-connected open surface $\mathcal{S}$ and a prescribed Beltrami coefficient $\mu$, we first partition $\mathcal{S}$ into several smaller subdomains. Then, we compute the free-boundary conformal maps from them to $\mathbb{R}^2$ in parallel. We then compose a free-boundary quasi-conformal map with the given Beltrami coefficient for each subdomain in parallel. As computing quasi-conformal maps on several small subdomains in a parallel way is much more efficient than computing the global quasi-conformal map directly, our algorithm is more efficient than many existing global parameterization methods. After computing the initial maps, we utilize an idea called partial welding~\cite{choi2020parallelizable} to glue the boundaries of the flattened subdomains along their common arcs. In particular, since $\mathcal{S}$ is a multiply-connected surface, we propose a variant of the original partial welding algorithm in~\cite{choi2020parallelizable} to achieve this task. Moreover, in order to transform the boundaries to circles, we propose a parallel version of the Koebe's iteration~\cite{koebe1910konforme} that is highly compatible with our algorithm. By the composition formula of Beltrami coefficients in quasi-conformal theory~\cite{gardiner2000quasiconformal}, the partial welding procedure and the Koebe's iteration will not induce any change in the prescribed Beltrami coefficient $\mu$ as every function involved in these steps is conformal. The computation of partial welding and the parallel Koebe's iteration relies on a method called the geodesic algorithm~\cite{marshall2007convergence}, whose convergence is theoretically guaranteed under certain mild conditions. All the computations in these two steps only involve the boundary points and welding paths and hence are highly efficient. Finally, using the new boundary conditions generated by the above procedures, we obtain the global quasi-conformal parameterization by solving the Laplace equation on each subdomain in parallel. 

The rest of this paper is organized as follows. In Section~\ref{sec:2}, we review the previous works on surface parameterization, with an emphasis on conformal and quasi-conformal parameterizations. In Section~\ref{sec:3}, we introduce the mathematical concepts related to this work. In Section~\ref{sec:4}, we describe our proposed method for the global conformal and quasi-conformal parameterizations of multiply-connected surfaces. In Section~\ref{sec:5}, we present experimental results and comparisons with other methods to demonstrate the effectiveness of our method. In Section~\ref{sec:6}, we show several applications of our proposed method in different fields. In Section~\ref{sec:7}, we discuss the limitations of our method and outline possible future research directions.

\section{Related works}
\label{sec:2}
In the past few decades, surface parameterization has attracted tremendous research attention in the area of geometry processing, graphics and vision. Detailed surveys and reviews on the topic can be found in~\cite{floater2005surface,sheffer2006mesh,hormann2007mesh}. In particular, since it is in general impossible to achieve isometric (both area-preserving and angle-preserving) parameterizations except for surfaces with zero Gaussian curvature, two major types of surface parameterization methods are the area-preserving parameterizations and the angle-preserving parameterizations. 

Existing area-preserving parameterization methods include the locally authalic map~\cite{desbrun2002intrinsic}, Lie advection~\cite{zou2011authalic}, optimal mass transport (OMT)~\cite{zhao2013area,cui2019spherical,giri2020open}, density-equaling map (DEM)~\cite{choi2018density,choi2020area} and stretch energy minimization (SEM)~\cite{yueh2019novel}. Although the area structure of the input surface can be well-preserved by these methods, the angle structure is usually significantly distorted. Since the angle structure is closely related to the local geometry of the surface, the distortion in the angle structure may induce obstacles for some applications. In these situations, angle-preserving parameterizations may be more preferable.

Existing conformal parameterization methods for simply-connected open surfaces include least-squares conformal map (LSCM)~\cite{levy2002least}, discrete natural conformal parameterization (DNCP)~\cite{desbrun2002intrinsic}, angle-based flattening (ABF)~\cite{sheffer2001parameterization, sheffer2005abf++, zayer2007linear}, holomorphic 1-form~\cite{gu2003global}, discrete Yamabe flow~\cite{luo2004combinatorial, wang2007brain}, discrete Ricci flow~\cite{jin2008discrete, yang2009generalized, zhang2014unified}, fast disk conformal map~\cite{choi2015fast}, boundary first flattening~\cite{sawhney2017boundary}, linear disk conformal map~\cite{choi2018linear}, conformal energy minimization~\cite{yueh2017efficient}, parallelizable global conformal parameterization (PGCP)~\cite{choi2020parallelizable,choi2022free} and spherical cap conformal map~\cite{shaqfa2021spherical}. For simply-connected closed surfaces, existing spherical conformal parameterization methods include harmonic energy minimization~\cite{gu2004genus,lai2014folding} and its linearizations~\cite{angenent1999laplace, haker2000conformal,choi2015flash,choi2016spherical} and parallelizable global conformal parameterization (PGCP)~\cite{choi2020parallelizable}. While many surfaces in real applications may be multiply-connected, the conformal mapping of multiply-connected surfaces is less studied. Existing conformal mapping methods between multiply-connected planar domains include conformal welding~\cite{marshall2012conformal}, Schwarz--Christoffel map~\cite{crowdy2005schwarz,crowdy2007schwarz}, slit map~\cite{crowdy2006conformal}, and PlgCirMap~\cite{nasser2020plgcirmap}. For the conformal parameterization of multiply-connected surfaces, existing methods include the generalized Koebe's iteration~\cite{zeng2009generalized}, Laurent series~\cite{kropf2014conformal}, discrete conformal equivalence~\cite{bobenko2016discrete}, and poly-annulus conformal map (PACM) \cite{choi2021efficient}. 

Quasi-conformal maps are a generalization of conformal maps with bounded local geometric distortion. As they are less restrictive than conformal maps, there has been an increasing interest in quasi-conformal surface parameterization methods in recent years. Existing methods for computing quasi-conformal parameterization include auxiliary metric~\cite{zeng2009surface}, quasi-Yamabe flow \cite{zeng2012computing}, linear Beltrami solver (LBS)~\cite{lui2013texture,lam2014landmark,choi2016fast}, Beltrami holomorphic flow (BHF)~\cite{lui2012optimization,ng2014teichmuller}, QC iteration~\cite{lui2014teichmuller,meng2016tempo}, extremal quasiconformal map~\cite{weber2012computing}, bounded distortion map~\cite{lipman2012bounded,chien2016bounded}, discrete Beltrami flow~\cite{wong2014computation,wong2015computing}, quasi-conformal energy minimization (QCMC)~\cite{ho2016qcmc}, and least-squares quasi-conformal map (LSQC)~\cite{qiu2019computing}. In recent years, quasi-conformal maps have been used in various applications such as image and surface registration~\cite{lam2014landmark,yung2018efficient,qiu2020inconsistent} and shape analysis~\cite{choi2020tooth,choi2020shape}. 

\section{Mathematical background}
\label{sec:3}
\subsection{Quasi-conformal theory}
\label{sec:3.1}
\indent In this subsection, we briefly introduce quasi-conformal maps on the complex plane and on Riemann surfaces. For details, readers are referred to~\cite{astala2008elliptic,gardiner2000quasiconformal}. 

Quasi-conformal maps are a generalization of conformal maps and can be understood as maps with bounded conformality distortion. An orientation-preserving homeomorphism $f: \Omega \subset \mathbb{C} \to \Omega' \subset \mathbb{C}$ is said to be a \emph{quasi-conformal map} if it satisfies the Beltrami equation:
\begin{equation} 
\dfrac{\partial f}{\partial \bar{z}} = \mu_{f}(z)\dfrac{\partial f}{\partial z}, \label{eq:1}
\end{equation}
where $\mu_{f}(z)$ is a complex-valued Lebesgue-measurable function satisfying $\|\mu_{f}(z)\|_{\infty} < 1$ called the \emph{Beltrami coefficient} of $f$. $\mu_{f}(z)$ encodes the information about the conformality distortion of $f$. If $\mu_{f}(z) = 0$ for all $z$, then Equation~\eqref{eq:1} becomes the Cauchy--Riemann equation and hence $f$ is conformal. Geometrically, a quasi-conformal mapping maps infinitesimal circles to infinitesimal ellipses with eccentricity determined by the Beltrami coefficient (see Fig.~\ref{fig:distortion_inf}). 

\begin{figure}[t]
\centering
  \includegraphics[width=0.8\linewidth]{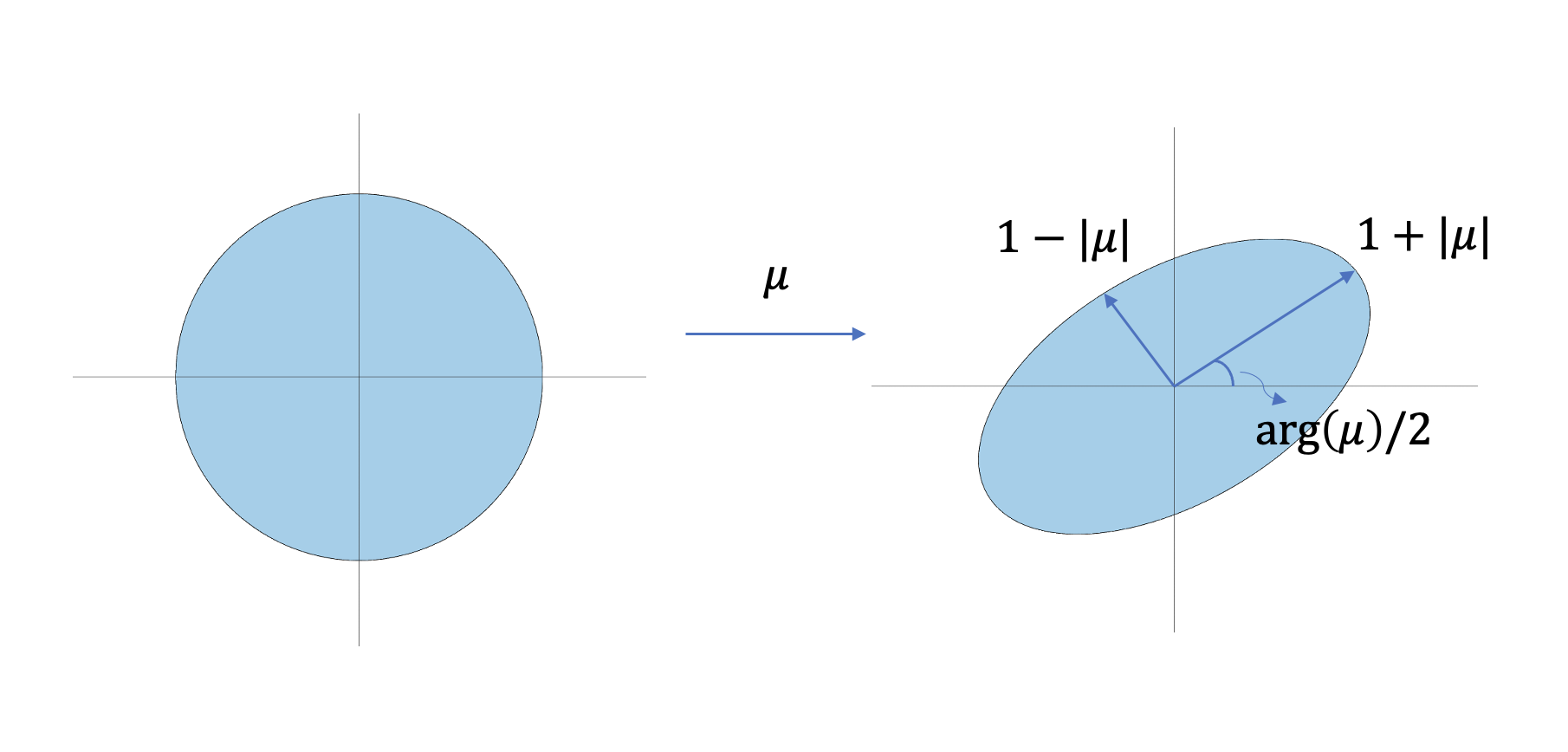}
\caption{An illustration of how the Beltrami coefficient determines the conformality distortion at the infinitesimal level, i.e. the differential map at a point associated with Beltrami coefficient $\mu$.}
\label{fig:distortion_inf}
\end{figure} 

The following theorem by Ahlfors and Lars, called the \emph{Measurable Riemann Mapping Theorem}~\cite{ahlfors1960riemann}, is a generalization of the Riemann Mapping Theorem for conformal maps to the case of quasi-conformal maps.
\begin{theorem} (Measurable Riemann Mapping Theorem)  Suppose $\mu:\mathbb{C} \to \mathbb{C}$ is Lebesgue measurable and satisfies $\|\mu\|_{\infty} < 1$. Then there is a quasi-conformal homeomorphism $\phi$ from $\mathbb{C}$ onto itself, which is in the Sobolev space $W^{1,2}(\mathbb{C})$ and satisfies the Beltrami equation~\eqref{eq:1} in the distribution sense. Furthermore, by fixing 0, 1, and $\infty$, the associated quasi-conformal homeomorphism $\phi$ is uniquely determined.
\end{theorem}

Conversely, given an orientation-preserving homeomorphism $\phi$, we can compute its Beltrami coefficient $\mu_{f}$ using the Beltrami equation~\eqref{eq:1}:
\begin{equation}
    \mu_{\phi}(z) = \dfrac{\partial \phi}{\partial \bar{z}}/\frac{\partial \phi}{\partial z}.
\end{equation}
This gives the following relation between the Jacobian $J_{\phi}$ and the Beltrami coefficient $\mu_{\phi}$:
\begin{equation}
    J_{\phi}(z) = \abs{\dfrac{\partial \phi}{\partial z}}^2\bigg(1 - \mu_{\phi}(z)\bigg)^2.
    \label{eq:2}
\end{equation}
Note that $J_{\phi}(z)>0$ everywhere as $\phi$ is an orientation-preserving homeomorphism, and hence we must have $\abs{\mu_{\phi}(z)} < 1$ for all $z$. By the measurable Riemann mapping theorem and the above observation, we conclude that there is a one-one correspondence between quasi-conformal homeomorphisms and Beltrami coefficients strictly less than 1. 

Moreover, we have the following composition formula for the Beltrami coefficient of a composition of two quasi-conformal maps. Suppose $f,g:\mathbb{C}\to \mathbb{C}$ are quasi-conformal maps with Beltrami coefficients $\mu_{f}$ and $\mu_{g}$ respectively. Then, the Beltrami coefficient of $g\circ f$ is
\begin{equation}
    \mu_{g\circ f} = \dfrac{\mu_{f} + (\mu_{g}\circ f)\tau}{1+\bar{\mu}_{f}(\mu_{g}\circ f)\tau},\ \ \tau = \frac{\bar{f}_z}{f_z}. \label{eq:3}
\end{equation}
In particular, if $g$ is a conformal map, we have $\mu_{g\circ f} = \mu_{f}$. In other words, given a quasi-conformal map $f$ with Beltrami coefficient $\mu$, the Beltrami coefficient of $g\circ f$ is always $\mu$ for any conformal map $g$. This observation plays an important role in our proposed algorithm for multiply-connected quasi-conformal parameterization. 

The following theorem relates the regularity of a quasi-conformal map with its Beltrami coefficient~\cite{astala2008elliptic}:
\begin{theorem} 
Suppose $f \in W_{loc}^{1,2}(\mathbb{C},\mathbb{C})$ is the solution to the Beltrami equation~\eqref{eq:1}, where the Beltrami coefficient $\mu(z) \in C_{loc}^{l,\alpha}(\mathbb{C},\mathbb{C})$, $\norm{\mu}_{\infty}<1$. Then, $f\in C_{loc}^{l+1,\alpha}(\mathbb{C},\mathbb{C})$. 
\end{theorem}

For quasi-conformal maps of Riemann surfaces, one can generalize the concept of Beltrami coefficients to Beltrami differentials via the local charts of the surfaces. More specifically, the \emph{Beltrami differential} $\mu(z)\dfrac{d\bar{z}}{dz}$ on a Riemann surface $S$ is an assignment to each chart $(U_{\alpha},\phi_{\alpha})$ of an $L_{\infty}$ complex-valued function $\mu_{\alpha}$, defined on local parameters $z_{\alpha}$, such that 
\begin{equation}
    \mu_{\alpha}(z_{\alpha})\dfrac{d\bar{z}_{\alpha}}{dz_{\alpha}} = \mu_{\beta}(z_{\beta})\dfrac{d\bar{z}_{\beta}}{dz_{\beta}} \label{eq:4}
\end{equation}
on the domain which is also covered by another chart $(U_{\beta},\phi_{\beta})$, where $\dfrac{dz_{\beta}}{dz_{\alpha}} = \dfrac{d}{dz_{\alpha}}\phi_{\alpha\beta}$ and $\phi_{\alpha\beta} = \phi_{\beta}\circ \phi_{\alpha}^{-1}$. In particular, if a surface can be covered by a single chart, we can use the Beltrami coefficient defined on that chart to represent the Beltrami differential of the surface. As our work focuses on multiply-connected open surfaces, we can simply find a free boundary conformal map from the given surface onto $\mathbb{C}$ and use that as the global chart to represent the Beltrami differential. Therefore, the Beltrami coefficient and the Beltrami differential are used interchangeably in our method. 

\subsection{Variational formulation of quasi-conformal map}
\label{sec:3.2}
Here we introduce a variational approach called the least-squares quasi-conformal map (LSQC), developed by Qiu \textit{et al.}~\cite{qiu2019computing}, for solving the Beltrami equation~\eqref{eq:1} to get free-boundary quasi-conformal maps. The formulation is an analog of the DNCP/LSCM formulation~\cite{desbrun2002intrinsic,levy2002least} for free-boundary conformal maps of 2D domains. Suppose $f: \Omega  \subset \mathbb{C} \to \Omega'  \subset \mathbb{C}$ is a quasi-conformal map. We write $f = u + iv$ and $\mu_{f} = \rho + i\tau$, where $u$, $v$, $\rho$ and $\tau$ are real-valued functions. Also, let 
\begin{equation}
A = \frac{1}{1-\abs{\mu}^2}
\begin{pmatrix}
(\rho-1)^2 + \tau^2 & -2\tau \\
-2\tau & (1+\rho)^2 + \tau^2
\end{pmatrix}.
\label{eqt:A}
\end{equation}
From the above, we can transform the Beltrami equation~\eqref{eq:1} into
\begin{equation}
    \begin{pmatrix}u_x \\ u_y \end{pmatrix} =
    \begin{pmatrix} 0  &  1 \\ -1  &  0 \end{pmatrix} A
    \begin{pmatrix} v_x \\ v_y \end{pmatrix}.
\end{equation}
Then, using the relation $u_{xy} = u_{yx}$, we obtain the following equation
\begin{equation}
    \nabla \cdot (A\nabla v(z)) = 0.  \label{eq:5}
\end{equation}
Similarly, we can express $v_x, v_y$ in terms of $u_x, u_y$ and get
\begin{equation}
    \nabla \cdot (A\nabla u(z)) = 0.  \label{eq:6}
\end{equation}
It can be observed that the two equations above are the Euler--Lagrange equations of the following two Dirichlet type energies respectively:
\begin{equation}
    E_{A}(u)=\dfrac{1}{2}\int_{\Omega}\norm{A^{1/2}\nabla u}^2 dxdy,\ \ \ 
    E_{A}(v)=\dfrac{1}{2}\int_{\Omega}\norm{A^{1/2}\nabla v}^2 dxdy.  \label{eq:7}
\end{equation}
Note that Equations \eqref{eq:5} and \eqref{eq:6} are necessary conditions of $u$ and $v$ derived from the Beltrami equation~\eqref{eq:1}. One can also define the following \emph{least-squares quasi-conformal energy} using the Beltrami equation~\eqref{eq:1} directly:
\begin{equation}
    E_{QC}^{\mu}(u,v) = \dfrac{1}{2}\int_{\Omega}\norm{P\nabla u + JP\nabla v}^2 dx dy,  \label{eq:8}
\end{equation}
where
\begin{equation}
    P = \dfrac{1}{\sqrt{1-\abs{\mu}^2}}
    \begin{pmatrix} 1-\rho & -\tau \\ -\tau & 1+\rho \end{pmatrix} \ \ \text{ and } \ \ 
    J = \begin{pmatrix} 0 & -1 \\ 1 & 0 \end{pmatrix}.
\end{equation}
Since $P^{T}P = A$, the Beltrami equation~\eqref{eq:1} holds if and only if $E_{QC}^{\mu}(u,v) = 0$.
The following equation relates the three energies $E_{A}(u)$, $E_{A}(v)$ and $E_{QC}^{\mu}(u,v)$:
\begin{equation}\label{eq:9}
    E_{A}(u) + E_{A}(v) - E_{QC}^{\mu}(u,v) = \mathcal{A}(u,v) = \int_{\Omega}(u_x v_y - v_x u_y)dx dy.
\end{equation}
Since $f$ is an orientation-preserving homeomorphism, $\mathcal{A}(u,v)$ is the area of $\Omega' = f(\Omega)$. For this reason, $\mathcal{A}(u,v)$ is called the \emph{area functional}. Now, since $E_{QC}^{\mu}(u,v)$ is always positive, we have the following inequality:
\begin{equation}
    E_{A}(u) + E_{A}(v) \geq \mathcal{A}(u,v). \label{eq:10}
\end{equation}
The equality holds if and only if $E_{QC}^{\mu}(u,v) = 0$, i.e. the Beltrami coefficient of $f = u + iv$ is equal to $\mu$.

Later on, we will see that $E_{A}(u), E_{A}(v)$ and $\mathcal{A}(u,v)$ can all be efficiently computed numerically, which allows us to compute \emph{free-boundary quasi-conformal maps} efficiently. 

On the other hand, if we want to compute a map from a domain to some specific domain such as a disk or a rectangle, we need to specify the boundary conditions. Suppose we want to compute a conformal map $f$ from a simply-connected domain $\Omega$ to the unit disk $\mathbb{D}$. By the measurable Riemann mapping theorem, $f$ is unique up to a M\"obius transformation. Therefore, the boundary condition $F(\partial\Omega)$ should be carefully set; otherwise, such a quasi-conformal map may not exist. To get the admissible boundary condition, we are going to use the geodesic algorithm developed by Marshall~\cite{marshall2007convergence}, which will be introduced later in this paper. By the elliptic PDE theory, with the admissible boundary condition, Equations~\eqref{eq:5} and \eqref{eq:6} have a unique solution and yield the desired quasi-conformal map.

Finally, we remark that since conformal maps are a special case of quasi-conformal maps with $\mu \equiv 0$, and the above results also hold for conformal maps and are consistent with the results in the conformal mapping literature~\cite{hutchinson1991computing,pinkall1993computing}.

\subsection{Conformal welding}
\label{sec:3.3}
There are several equivalent ways to describe the conformal welding problem. Here, we adopt the version in~\cite{marshall2012conformal,sharon20062d}. Let $\bar{\mathbb{C}} := \mathbb{C} \cup \{\infty\}$ denote the extended complex plane and $\mathbb{D}$ denote the unit disk $\{z \in \mathbb{C}:\abs{z} \leq 1\}$. Given an increasing homeomorphism $h$ of $\partial \mathbb{D}$, the conformal welding problem is to find a Jordan curve $J$ and two conformal maps $f,g$ such that $f$ and $g$ map $\mathbb{D}$ and $\bar{\mathbb{C}} \backslash \textup{int}(\mathbb{D})$ to $J \cup \textup{int}(J)$ and $J \cup \textup{ext}(J)$ respectively, where  $\textup{int}(J)$ and $\textup{ext}(J)$ are the interior and exterior of $J$ respectively, and $f(h(x)) = g(x)$ on $\partial \mathbb{D}$. Since there exists a conformal map between the unit disk and the upper half plane, we can also formulate the problem for the upper and lower half planes with an increasing homeomorphism on the real axis.

The conformal welding problem may not have a solution for a general homeomorphism $h$. However, the existence of conformal welding can be proved if $h$ satisfies certain conditions. Here, we introduce the notion of quasi-symmetric functions. Let $h$ be a continuous, strictly increasing function defined on an interval $I$ of the x-axis. We call $h$ \emph{k-quasi-symmetric} on $I$~\cite{lehto1973quasiconformal} if there exists a positive constant $k$ such that 
\begin{equation}
    \dfrac{1}{k} \leq \dfrac{h(x+t) - h(x)}{h(x) - h(x-t)} \leq k,
\end{equation}
for all $x,x-t \in I$ with $t > 0$.

The following theorem shows the solvability of the conformal welding problem when $h$ is a quasi-symmetric homeomorphism of the real axis. 
\begin{theorem} (Sewing theorem~\cite{lehto1973quasiconformal}).
Let $h$ be a quasi-symmetric function on the real axis. Then the upper and lower half-planes can be mapped conformally onto disjoint Jordan domains $D,\Omega$ by two maps $\phi,\phi^*$, such that $\phi(x) = \phi^*(h(x))$ for all $x\in \mathbb{R}$.
\label{thm:sewing}
\end{theorem}
The proof of the above theorem is based on approximation techniques of quasi-symmetric functions. The solvability of the conformal welding problem can also be proved using the existence of solutions to the Beltrami equation as shown by Pfluger~\cite{pfluger1960ueber}.

\subsection{Geodesic algorithm}
\label{sec:3.4}
The Riemann mapping theorem guarantees the existence of a conformal map from a simply-connected open subset of $\mathbb{C}$ to the unit disk, unique up to a M\"{o}bius transformation. However, this theorem does not provide a way to compute such a conformal map explicitly. In the 1980s, K\"{u}hnau~\cite{kuhnau1983numerische}, and Marshall and Morrow~\cite{Marshall1987compositions} independently proposed the \emph{zipper algorithm} for computing conformal maps from a simply-connected open set to the unit disk. Later, Marshall and Rohde~\cite{marshall2007convergence} proved the convergence in different cases for a variant of the zipper algorithm called \emph{geodesic algorithm}. As described by Marshall and Rohde~\cite{marshall2007convergence}, the geodesic algorithm can be viewed as an approximate solution to a conformal welding problem or as a discretization of the Loewner differential equation. The details, variants and convergence of the geodesic algorithm can be found in~\cite{marshall2007convergence}. Below, we briefly introduce the geodesic algorithm.

The key ingredient of the geodesic algorithm is the two-fold map shown in Fig.~\ref{fig:geodesic}, which is a composition of a M\"obius transformation, a square map, and a square root map. In one direction, it maps a hyperbolic geodesic to the real axis. Given $z_1$ on the upper half plane, we denote by the red line $\gamma$ the circular arc from $0$ to $z_1$, which is a hyperbolic geodesic. The map $f_{z_1}$ conformally maps $\gamma$ to $[0,z_3]$ or $[-z_3,0]$ depending on the choice of the branch for the square root map. The rest of the upper half plane $\mathbb{H}\backslash\gamma$ is conformally mapped to $\mathbb{H}$. In the reverse direction, note that two line segments $[-z_3,0]$ and $[0,z_3]$ are both mapped to $[0,z_3^2]$ by a square map, and eventually mapped to the curve $\gamma$. Hence, this direction allows us to conformally align two different lines, which can then be used to compute conformal welding.

\begin{figure}[t]
\centering
\includegraphics[width=0.9\linewidth]{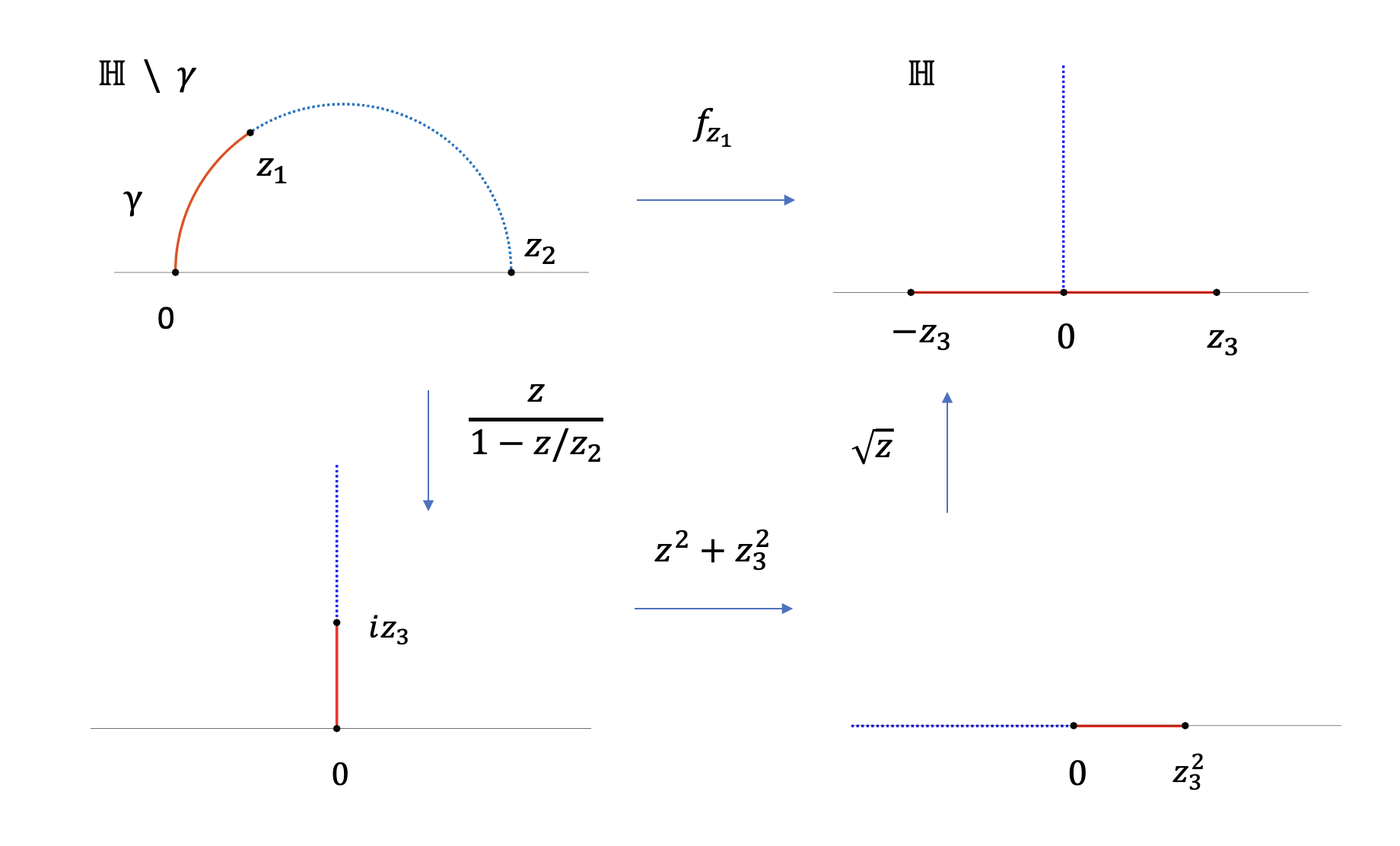}
\caption{The basic conformal map $f_{z_1}$ of the geodesic algorithm in~\cite{marshall2007convergence}.}
\label{fig:geodesic}
\end{figure}

In order to compute a Riemann mapping from some Jordan domain $\Omega$ to $\mathbb{H}$ by the geodesic algorithm, we only need a sequence of boundary points $\{z_0,z_1,\dots,z_n\}$ of $\partial \Omega$ that are sufficiently dense on $\partial \Omega$. The starting map is given by 
 \begin{equation}
     g_1(z) = i\sqrt{\dfrac{z-z_1}{z-z_0}},
 \end{equation}
with $g_1(z_1) = 0$ and $g_1(z_0) = \infty$. Let $\xi_2 = g_1(z_2)$ and $g_2 = f_{\xi_2}$, where $f_{\xi_2}$ is the map defined in Fig.~\ref{fig:geodesic}. We repeat this process for all the boundary points to get
\begin{equation}
    \xi_k = g_{k-1}\circ g_{k-2}\circ\cdots\circ g_1(z_k)
\end{equation}
and 
\begin{equation}
    g_k = f_{\xi_k}
\end{equation}
for $k = 2,\dots,n$. We then compute a final map by defining
\begin{equation}
    \xi_{n+1} = g_n\circ\cdots\circ g_1(z_0)\in\mathbb{R}
\end{equation}
and
\begin{equation}
    g_{n+1} = \pm\bigg(\dfrac{z}{1-z/\xi_{n+1}}\bigg)^2,
\end{equation}
where the positive sign is chosen when the data points are in anti-clockwise orientation, and the negative sign otherwise. 
The composition mapping
$g = g_{n+1}\circ g_n\circ\cdots g_1$
gives a conformal map from $\Omega$ to $\mathbb{H}$. Although originally invented to be in this form, as indicated by Marshall in~\cite{marshall2007convergence}, the computation of the mapping is more reliable when we perform it on the right half plane instead of the upper half plane due to the default choice of branching in scientific computing software. In our algorithm, we perform all the computation on the right half plane. 

The convergence of the geodesic algorithm was proved in~\cite{marshall2007convergence}. In particular, under different assumptions on the regularity of the region $\Omega$, different convergence results can be established.

\subsection{Riemann mapping theorem for multiply-connected domains}
\label{sec:3.5}
While the Riemann mapping theorem focuses on the conformal equivalence between any simply-connected region in the complex plane and the open unit disk, there is also a generalization of this result to multiply-connected domains. Here, we present a result given in Chapter 17 of~\cite{henrici1993applied}, which shows that any region $R$ of connectivity $n \geq 2$ can be conformally mapped to the complement of $n$ closed circular disks. Such a region is called a \emph{circular region of connectivity $n$}.
\begin{theorem}
Let $R$ be a region of connectivity $n\geq 2$ in the extended complex plane with $\infty \in R$. Then, there exists a unique circular region of connectivity $n$ and a unique one-to-one analytic function $f$ satisfying $f(z) = z + O(\dfrac{1}{z})$ such that $f(R)=C$.
\end{theorem}
The book \cite{henrici1993applied} gives a constructive proof of this theorem, which was originally due to Koebe and hence called the \emph{Koebe's iteration}~\cite{koebe1910konforme}. We explain the Koebe's iteration in detail here as it is closely related to our proposed algorithm in this paper. Suppose the components of complements of $R$ are $K_1,K_2,\dots,K_n$. Let $R_0 := R, D_{0,i}:=K_i, i = 1,2,\dots,n$. Suppose in the $(k-1)$-th iteration, we have obtained a region $R_{k-1}$ of connectivity $n$, whose complements are $D_{k-1,i},i = 1,2,\dots,n$. Then, in the $k$-th iteration, let $j = k\ \text{mod}\ n,1\leq j \leq n$. We find the unique conformal map $h_k$, normalized at $\infty$, from $R_{k-1}\backslash D_{k-1,j}$ to the exterior of a disk. Let $D_{k,j}$ be that disk, and
\begin{equation}
R_k := h_{k}(R_{k-1}),\ D_{k,i} := h_k(D_{k,i-1}), \ i=1,2,\dots,n, \ i\neq j.
\end{equation}
Clearly, $R_k$ is a region of connectivity $n$ and the components of complements of it are $D_{k,i}, i=1,2,\dots,n$. The Koebe's algorithm goes cyclically on $i=1,2,\dots,n$, each time mapping one boundary component to a circle until the result converges.

Let $f$ denote the desired Riemann mapping, $f_k := h_{k}\circ h_{k-1}\circ \cdots \circ h_1$ and $g_k := f_k \circ f^{-1}$. We have the following estimate of the convergence rate~\cite{henrici1993applied}:
\begin{theorem}
\label{thm:Koebe}
There exists constants $\gamma > 0$ and $0<\mu <1$ such that for $k = 1,2,\dots$ and for all $w\in C$,
\begin{equation}
    \abs{g_{k}(w)-w} \leq \gamma \mu^{4[k/n]}.
\end{equation}
\end{theorem}
Numerically, in each iteration, we apply the geodesic algorithm to transform one of the boundaries to a circle and also update the coordinates of other boundaries~\cite{marshall2012conformal}. In practice, we find that the algorithm exhibits fast convergence, and usually we can already obtain a satisfactory result after performing only one iteration for each boundary. One example can be found in Fig.~\ref{fig:koebe}. As we shall see later, that is part of the reason why our proposed parallel Koebe's iteration method works.

Below, we also state the extension of the Riemann mapping theorem for multiply-connected domains to quasi-conformal maps presented in the book~\cite{lawrynowicz2006quasiconformal}.
\begin{theorem}
\label{thm:qcm}
Let $D$ be the closure of a domain bounded by $n$ disjoint Jordan curves. Suppose $\mu$ is a measurable function defined in $D$ and $\norm{\mu}_{\infty}<1$. Then, there exists a closed canonical circular domain $D'$ of connectivity $n$ and a solution $f$ to the Beltrami equation~\eqref{eq:1}, which represents a quasi-conformal homeomorphism of $D$ onto $D'$, determined uniquely up to conformal maps of $D'$ onto itself. 
\end{theorem}

In practice, given a multiply-connected domain $D$ with a prescribed Beltrami coefficient $\mu$, we can first compute a free-boundary quasi-conformal parameterization of it onto a domain $D_1$. After that, we compute the conformal map from $D_1$ to a circular domain $D_2$ using the Koebe's iteration. The composition of these two maps gives the desired result. 

\begin{figure}
\centering
\begin{subfigure}[t]{.5\textwidth}
\centering
\includegraphics[width=.6\linewidth]{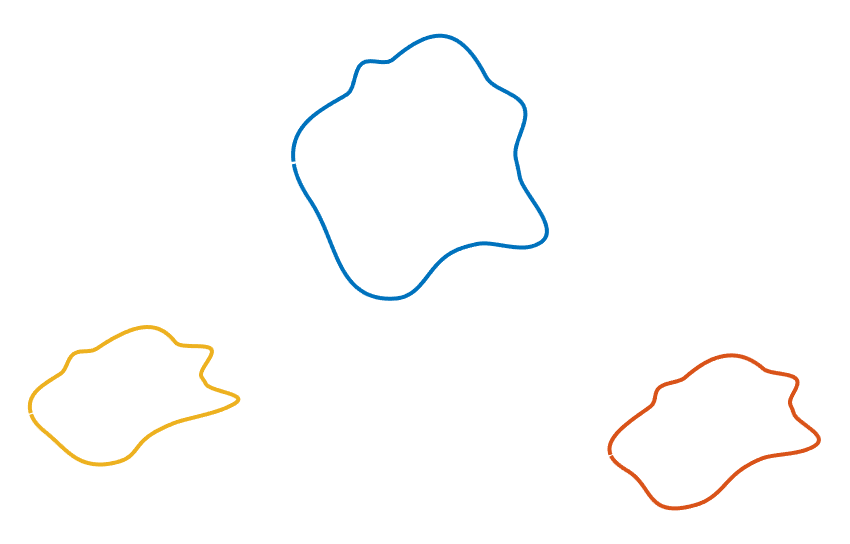}
\caption{The given boundary curves}
\label{fig:koebe1}
\end{subfigure}%
\begin{subfigure}[t]{.5\textwidth}
\centering
\includegraphics[width=.6\linewidth]{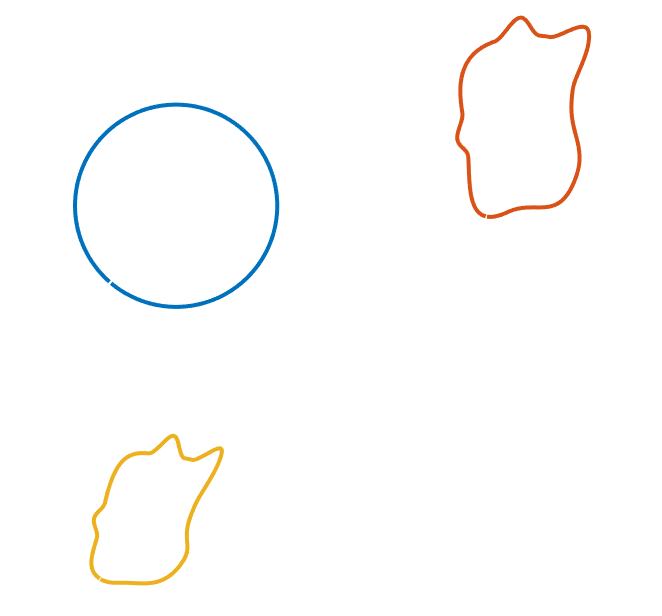}
\caption{First iteration}
\end{subfigure}
\begin{subfigure}[t]{.5\textwidth}
\centering
\includegraphics[width=.6\linewidth]{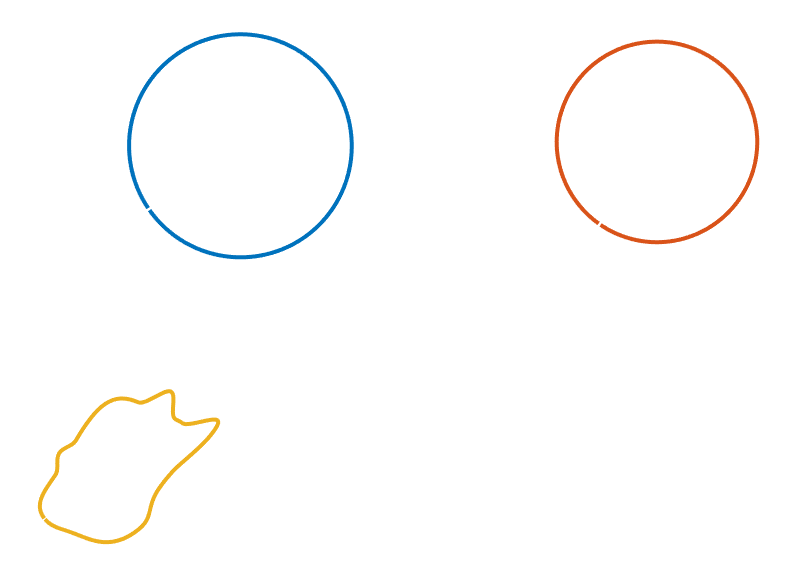}
\caption{Second iteration}
\label{fig:partition_multi}
\end{subfigure}%
\begin{subfigure}[t]{.5\textwidth}
\centering
\includegraphics[width=.6\linewidth]{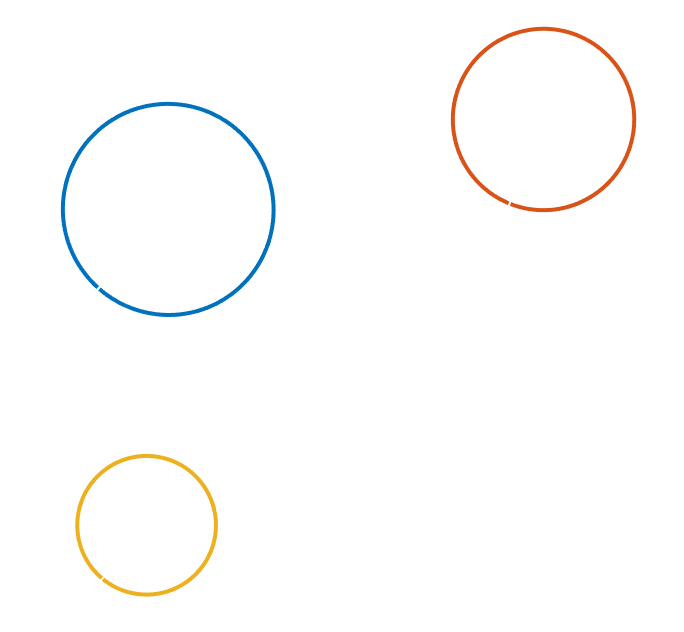}
\caption{Third iteration}
\label{fig:koebe4}
\end{subfigure}%
\caption{An example illustrating the fast convergence of the Koebe's iteration method, with each map computed using the geodesic algorithm.}
\label{fig:koebe}
\end{figure}

\section{Proposed method}
\label{sec:4}
\subsection{An overview of our proposed method}
\label{sec:4.1}
Let $\mathcal{S}$ be a multiply-connected surface in $\mathbb{R}^3$ represented by a triangle mesh $(\mathcal{V},\mathcal{F})$, where $\mathcal{V}$ denotes the set of vertices and $\mathcal{F}$ denotes the set of faces. Given a target Beltrami coefficient $\mu$, we aim to compute the global quasi-conformal parameterization of $\mathcal{S}$ onto the unit disk with circular holes efficiently and accurately. 

First, we partition the entire mesh $\mathcal{S}$ into multiple submeshes $\mathcal{S}_i, i=1,\dots,m$ such that each submesh is either simply-connected or multiply-connected with 1 inner hole of $\mathcal{S}$. Note that we may further partition the submeshes with 1 hole into more smaller simply-connected submeshes if necessary. Then, we compute a free-boundary conformal parameterization $\varphi_i^c:\mathcal{S}_i \to \mathbb{R}^2$ of each submesh onto the plane respectively. Here, we compute the conformal parameterization first because the Beltrami differentials on the surface depend on the choice of isothermal local charts, and the computed free-boundary conformal parameterization serves well in this role for the submeshes. The next step is to compute a free-boundary quasi-conformal map of each flattened submesh $\varphi_i^{qc}: \varphi_i^c(\mathcal{S}_i) \to \mathbb{R}^2$ based on the prescribed Beltrami coefficient, so that the composition $\varphi_i = \varphi_i^{qc} \circ \varphi_i^c$ gives the free-boundary quasi-conformal parameterization for every submesh. Note that both the conformal parameterization and quasi-conformal mapping steps are highly parallelizable as the computations for different submeshes are independent. Since all the remaining steps only involve conformal transformations, by the composition formula~\eqref{eq:3}, the Beltrami coefficient will be preserved by the remaining steps. We apply the geodesic algorithm to transform all the inner holes of the submeshes into circles. This step can be understood as a parallelizable version of the Koebe's iteration. We then apply the welding algorithm to obtain the desired boundary conditions of all submeshes. Note that the inner boundaries after welding are highly circular, as will be illustrated both theoretically and experimentally in the following sections. Finally, we solve the Laplace equation with the updated boundary conditions to obtain the quasi-conformal parameterization for each submesh, all of which together form the desired global quasi-conformal parameterization seamlessly (see Fig.~\ref{fig:overview} for an illustration).

\subsection{Surface partition}
\label{sec:4.2}
We first partition the given multiply-connected surface $\mathcal{S}$ into multiple submeshes $\mathcal{S}_i, i=1,\dots,m$, which can be done by existing mesh partitioning algorithms or manually prescribing some edges as the partition cuts. Suppose $\mathcal{S}$ contains an outer boundary $\gamma_0$ and $k$ disjoint inner boundaries $\{\gamma_i\}_{i=1}^k$, where each $\gamma_i$ is represented by a set of boundary edges, our partition procedure consists of the following two steps. In the first step, we choose a set of cutting edges denoted by $\mathcal{E}_{cut}$ such that $\mathcal{E}_{cut}$ does not contain any boundary edges. The reason is that if we remove $\mathcal{E}_{cut}\bigcup \gamma_0$ from $\mathcal{S}$, we may obtain several subdomains that are disconnected from each other. Hence, we choose the partition $\mathcal{S} = \bigcup_{i=1}^{m}\mathcal{S}_i$ by assigning $\mathcal{S}_{i}$ to be each of the components. In other words, we obtain $\mathcal{S}_1 = (\mathcal{V}_1,\mathcal{F}_1), \mathcal{S}_2 = (\mathcal{V}_2,\mathcal{F}_2), \dots, \mathcal{S}_m = (\mathcal{V}_{m},\mathcal{F}_{m})$. Mathematically, the following conditions should be satisfied:
\begin{equation}
\mathcal{E}_{cut} \bigcap \gamma_i = \emptyset\ \text{for all}\ i = 1,\dots,k,
\end{equation}
and
\begin{equation}
\mathcal{S}_i \bigcap \mathcal{S}_j \subset \mathcal{E}_{cut}\ \text{or}\ \mathcal{S}_i \bigcap \mathcal{S}_j=\emptyset \ \text{for all}\ i,j = 1,\dots,m.
\end{equation}
Here, we restrict all $\mathcal{S}_i$ to be simply-connected or multiply-connected with only 1 inner hole. Such a restriction reduces the difficulty of computing partial welding for multiply-connected meshes and performing the parallel Koebe's iteration, as will be explained later. In the second step, we can further partition the submeshes into smaller meshes if necessary and possible. For example, if a submesh $\mathcal{S}_i$ with 1 hole is still a large mesh, we can partition it into several simply-connected meshes.

\subsection{Free-boundary quasi-conformal parameterization of the submeshes}
\label{sec:4.3}
After getting the submeshes $\mathcal{S}_i, i=1,\dots,m$, we compute a free-boundary conformal parameterization of each of them onto the plane followed by a free-boundary quasi-conformal map using the variational formulation in Section~\ref{sec:3.2} by a finite element approach. 

The numerical computation of the quasi-conformal mapping follows the implementation described in~\cite{qiu2019computing}. Given a flattened triangle mesh $\Omega$ represented by a set of vertices $\{w_1,w_2,\dots,w_n\}$ and a set of triangle faces, we discretize the prescribed Beltrami coefficient $\mu$ on $\Omega$ by assuming that $\mu$ is \emph{piecewise constant} on each triangle face, i.e., $\mu = \mu_{T}$ for some constant $\mu_{T}$ on each triangle face $T$ of $\Omega$. We aim to compute a map $f = u + iv:\Omega \rightarrow \tilde{\Omega}$ where $\tilde{\Omega}$ is a triangle mesh with same connectivity as $\Omega$ such that $f$ satisfies the Beltrami equation~\eqref{eq:1} in the sense that $f$ is \emph{piecewise linear} on each face $T$ and $\mu_{f}|_{T} = \mu_{T}$ for each face $T$. We denote the vertices of $\tilde{\Omega}$ by $\{f(w_1),f(w_2),\dots,f(w_n)\} = \{u_1+iv_1,u_2+iv_2,\dots,u_n+iv_n\}$. Let $u = \begin{pmatrix} u_1 & u_2 & \cdots & u_n \end{pmatrix}^T$ and $v = \begin{pmatrix} v_1 & v_2 & \cdots & v_n \end{pmatrix}^T$.

We then discretize the energies $E_{A}(u)=\int_{\Omega}\norm{A^{1/2}\nabla u}^2$ and $E_{A}(v)=\int_{\Omega}\norm{A^{1/2}\nabla v}^2$ in Equation~\eqref{eq:7} in the following way. Let $T$ be an arbitrary triangle with vertices $[w_0^T,w_1^T,w_2^T]$. Suppose the image of $T$ under $f$ is $[f(w_0^T),f(w_1^T),f(w_2^T)]=[u_0^T+iv_0^T,u_1^T+iv_1^T,u_2^T+iv_2^T]$. Since $f$ is linear on $T$, we can express the gradient of $f$ as 
\begin{equation}
\label{eq:11}
\nabla u|_T = \dfrac{1}{2\text{Area}(T)}\begin{pmatrix}0&-1\\1&0\end{pmatrix}\sum_{i=0}^2u_i(w_{2+i}^{T}-w_{1+i}^{T}),
\end{equation}
and
\begin{equation}
\label{eq:12}
\nabla v|_T = \dfrac{1}{2\text{Area}(T)}\begin{pmatrix}0&-1\\1&0\end{pmatrix}\sum_{i=0}^2v_i(w_{2+i}^{T}-w_{1+i}^{T}).
\end{equation}
Since $\mu$ is piecewise constant on each face $T$, the matrix $A$ given by Equation~\eqref{eqt:A} is a constant matrix determined by $\mu_{T}$ on each $T$. We can then discretize $\int_{\Omega}\norm{A^{1/2}\nabla u}^2$ and $\int_{\Omega}\norm{A^{1/2}\nabla v}^2$ by summing over all faces. We then obtain two quadratic forms
\begin{equation}
    E_{A}(u) = u^T \mathcal{L}_{\mu} u \ \ \text{ and } \ \ E_{A}(v) = v^T \mathcal{L}_{\mu} v,
\end{equation}
where $\mathcal{L}_{\mu}$ is a symmetric matrix called the \emph{generalized Laplacian matrix}. Furthermore, using Equations~\eqref{eq:11} and~\eqref{eq:12}, we can discretize the area matrix $\mathcal{A}(u,v)=\int_{\Omega}u_{x}v_{y}-v_{x}u_{y}$ in Equation~\eqref{eq:10} as another quadratic form
\begin{equation}
    \mathcal{A}(u,v) = \begin{pmatrix} u^T & v^T\end{pmatrix} \begin{pmatrix} 0 & U \\ -U & 0 \end{pmatrix} \begin{pmatrix} u \\ v \end{pmatrix}
    \label{eqt:area_matrix}
\end{equation}
for some skew-symmetric matrix $U$. Let 
\begin{equation}
    M = \begin{pmatrix} \mathcal{L}_{\mu} & 0 \\ 0 & \mathcal{L}_{\mu} \end{pmatrix} - \begin{pmatrix} 0 & U \\ -U & 0 \end{pmatrix}.
    \label{eqt:M}
\end{equation}
Note that $M$ is symmetric. By Equation~\eqref{eq:9}, to obtain the desired free-boundary quasi-conformal map, it suffices to solve the equation
\begin{equation}
\label{eq:13}
     E_{\text{QC}}^{\mu}(u,v) = \begin{pmatrix} u^T & v^T\end{pmatrix} M \begin{pmatrix} u \\ v \end{pmatrix} = 0.
\end{equation}
We have the following theorem:
\begin{theorem}
The solution of Equation~\eqref{eq:13} is unique under scaling, rotation, and translation.
\end{theorem}
\begin{proof}
Since $M$ is symmetric, it suffices to solve the equation
\begin{equation}
    M\begin{pmatrix} u \\ v \end{pmatrix} = 0.
\end{equation}
Suppose we fix two arbitrary points from $\{f(w_1),f(w_2),\dots,f(w_n)\}$. Then, we need to solve
\begin{equation}
    B\begin{pmatrix} u \\ v \end{pmatrix} = b,
\end{equation}
for some matrix $B$ and vector $b$. A direct consequence of Proposition 2.13 in~\cite{qiu2019computing} is that the matrix $B$ is of full rank. Hence, we obtain a unique solution if two arbitrary points are fixed.

On the other hand, suppose $\begin{pmatrix} u_0 \\ v_0 \end{pmatrix}$ is one solution to Equation~\eqref{eq:13}. Then, it is easy to check for any $k\in \mathbb{R}$, we have
\begin{equation}
    M \begin{pmatrix} k u_0 \\ k v_0 \end{pmatrix} = 0.
\end{equation}
Also, for any $\theta \in [0,2\pi]$, let $u_1 = \cos\theta u_0 - \sin\theta v_0$ and $v_1 = \sin\theta u_0 + \cos\theta v_0$. We have
\begin{equation}
    M \begin{pmatrix} u_1 \\ v_1 \end{pmatrix} = 0.
\end{equation}
Further, notice that if we let $f(w_i) = (x_0,y_0)$ for some $x_0,y_0$ for all $i = 1,\dots,n$, then clearly $E_{A}(u)$, $E_{A}(v)$, and $\mathcal{A}(u,v)$ are all zero. As a result, for any $x,y\in \mathbb{R}$,
\begin{equation}
    M \begin{pmatrix} u_0 + x \\ v_0 + y \end{pmatrix} = 0.
\end{equation}
Suppose the unique solution we obtain by fixing $f(w_i)$ and $f(w_j)$ to $(x_i,y_i)$ and $(x_j,y_j)$ respectively is $\begin{pmatrix} u_0 \\ v_0 \end{pmatrix}$. We can transform $\begin{pmatrix} u_0 \\ v_0 \end{pmatrix}$ by scaling, rotation, and translation so that $f(w_s) = (x_s,y_s)$ and $f(w_t) = (x_t,y_t)$ for arbitrary $s,t,(x_s,y_s),(x_t,y_t)$. We denote the transformed data points by $\begin{pmatrix} u_0 \\ v_0 \end{pmatrix}$. Notice that this is exactly the unique solution we can obtain by fixing $f(w_s)$ and $f(w_t)$ to $(x_s,y_s)$ and $(x_t,y_t)$ respectively. This completes the proof.
\end{proof}

As for the boundary conditions for solving the linear system, we usually set the target positions of two boundary vertices that are far away from each other in $\mathcal{M}$ as $(0,0)$ and $(1,0)$ to control the scale of the free-boundary mapping result.

Since the Beltrami differential on a surface in $\mathbb{R}^3$ depends on the choice of local chart (see Section~\ref{sec:3.1}), we cannot directly apply this method to compute a quasi-conformal flattening of a surface. Instead, we need to first compute a free-boundary conformal flattening $\varphi^c$ of a surface $\mathcal{S}$ onto $\mathbb{R}^2$ and then apply the above method to get a free-boundary quasi-conformal map $\varphi^{qc}$ in $\mathbb{R}^2$. For the conformal flattening map $\varphi^c = (u,v): \mathcal{S} \to \tilde{\Omega}$, the Dirichlet energy can be discretized as
\begin{equation}
    E(u) + E(v) = \frac{1}{2}\int_{\mathcal{S}}(\norm{\nabla u}^2 + \norm{\nabla v}^2) = 
    \begin{pmatrix} u^T & v^T\end{pmatrix} \begin{pmatrix} \mathcal{L} & 0 \\ 0 & \mathcal{L} \end{pmatrix} \begin{pmatrix} u \\ v \end{pmatrix},
    \label{eqt:DNCP_energy}
\end{equation}
where $\mathcal{L}$ is the cotangent Laplacian matrix~\cite{pinkall1993computing}. The DNCP method~\cite{desbrun2002intrinsic} discretizes the area using an approach different from Equation~\eqref{eqt:area_matrix}. Specifically, by Green's theorem,
\begin{equation}
    \mathcal{A}(\varphi^c) = \int_{\tilde{\Omega}} dx\, dy = \frac{1}{2} \oint_{\partial \tilde{\Omega}} -y \, dx + x \, dy,
    \label{eqt:green}
\end{equation}
Therefore, in the simply-connected case which \cite{desbrun2002intrinsic} focuses on, the area is discretized as
\begin{equation}
    \mathcal{A}(\varphi^c) = \dfrac{1}{2}\sum\limits_{[w_i,w_j] \in \partial \mathcal{S}} (u_iv_j-u_jv_i) = \begin{pmatrix} u^T & v^T\end{pmatrix} Q \begin{pmatrix} u \\ v \end{pmatrix}
    \label{eqt:area_term}
\end{equation}
for some symmetric matrix $Q$. The free-boundary conformal parameterization $\varphi^c$ is then obtained by solving
\begin{equation}
\Bigg(\left(
\begin{array}{cc}
\mathcal{L} & 0 \\
0 & \mathcal{L}
\end{array}
\right) - Q\Bigg) \begin{pmatrix}
      u \\ v \end{pmatrix} = 0.
     \label{eqt:DNCP}
\end{equation}

In our case, some submeshes obtained from the partition step may be multiply-connected. To apply the DNCP formulation for parameterizing them, a nature extension of Equation~\eqref{eqt:area_term} for multiply-connected meshes is presented below. Let $\mathcal{S}$ be a multiply-connected mesh. Denote the outer boundary of it as $\gamma_0$ and the inner boundaries as $\gamma_1, \dots, \gamma_p$, where $p \geq 1$. The area $\mathcal{A}(\varphi^c)$ can then be discretized as
\begin{equation}
    \mathcal{A}(\varphi^c) = \mathcal{A}_0 - \mathcal{A}_1 - \cdots - \mathcal{A}_p,
    \label{eqt:area_matrix_multiply}
\end{equation}
where $\mathcal{A}_0, \dots, \mathcal{A}_p$ are the areas of the regions enclosed by $\gamma_0, \dots, \gamma_p$ respectively. Each of them can be computed using the formula in Equation~\eqref{eqt:area_term}. Since all terms are expressed using the corresponding boundary vertices in $\mathcal{S}$, the area $\mathcal{A}(\varphi)$ can again be written in the form $\begin{pmatrix} u^T & v^T\end{pmatrix} \tilde{Q} \begin{pmatrix} u \\ v \end{pmatrix}$ for some matrix $\tilde{Q}$. We can then replace $Q$ with $\tilde{Q}$ in Equation~\eqref{eqt:DNCP} and solve it to obtain the free-boundary conformal parameterization $\varphi^c$.
\begin{remark}
Careful checking reveals that the two approaches for discretizing the area functional in Equation~\eqref{eqt:area_matrix} and Equation~\eqref{eqt:area_matrix_multiply} in fact give us the same quadratic form for mappings in the plane and hence either of them can be used for the computation of the 2D quasi-conformal map $\varphi^{qc}$. In practice, Equation~\eqref{eqt:area_matrix} is a direct summation of energies over all faces, while Equation~\eqref{eqt:area_matrix_multiply} only involves the boundary vertices but requires the boundary edges to be extracted and in correct orientations.
\end{remark}

We summarize the procedure for the free-boundary quasi-conformal parameterization in Algorithm~\ref{alg:free}.

\begin{algorithm2e}[h!]
\KwInput{An open surface $\mathcal{S}_i$ with $p \geq 0$ inner holes and a Beltrami coefficient $\mu$.}
\KwOutput{A free-boundary quasi-conformal parameterization $\varphi_i:\mathcal{S}_i \to \mathbb{R}^2$.}
\SetKwBlock{Conformal}{Initial conformal parameterization step:}{}
\SetKwBlock{QC}{Quasi-conformal mapping step (if $\mu \neq 0$):}{}
\Conformal{
Compute the cotangent Laplacian matrix $\mathcal{L}$ of $\mathcal{S}_i$\;
Compute the area of $\mathcal{S}_i$ using Equation~\eqref{eqt:area_term} (if $p = 0$) or Equation~\eqref{eqt:area_matrix_multiply} (if $p \geq 1$)\;
Compute a free-boundary conformal parameterization $\varphi_i^{c}: \mathcal{S}_i \to \mathbb{R}^2$ by solving Equation~\eqref{eqt:DNCP}\;
}
\QC{
Compute the generalized Laplacian matrix $\mathcal{L}_{\mu}$\;
Compute the area matrix using Equation~\eqref{eqt:area_matrix} or Equation~\eqref{eqt:area_matrix_multiply}\;
Compute a free-boundary quasi-conformal map $\varphi_i^{qc}: \varphi_i^c(\mathcal{S}_i) \to \mathbb{R}^2$ by solving Equation~\eqref{eq:13};\\
}
The desired free-boundary quasi-conformal parameterization is given by $\varphi_i = \varphi_i^{qc} \circ \varphi_i^{c}$\;
\caption{Free-boundary quasi-conformal parameterization of simply-connected and multiply-connected open surfaces} 
\label{alg:free}
\end{algorithm2e}

\subsection{Partial welding}
\label{sec:4.4}
In the closed conformal welding problem introduced in Section~\ref{sec:3.3}, we are given a homeomorphism between the boundaries of two shapes and we need to glue the entire boundaries consistently. By contrast, in our problem we partition a mesh into several submeshes and compute the free-boundary quasi-conformal maps for them respectively, and hence we only need to conformally glue these submeshes along the partition paths to obtain the global quasi-conformal parameterization. Since the outer boundary edges are never contained in the partition paths, the gluing paths are just continuous subsets of the boundary of the submeshes. Therefore, we need to conformally glue two submeshes with respect to a homeomorphism between two partial arcs of their boundaries. To solve this problem, we extend the partial welding method developed in~\cite{choi2020parallelizable,choi2022free}, which is a variant of the geodesic algorithm designed for handling simply-connected surfaces. Below, we first briefly introduce the method for the simply-connected case and then describe how we can extend it for meshes with holes. 

\subsubsection{The simply-connected case}
The geodesic algorithm solves the closed welding problem by aligning the corresponding boundary points one-by-one. For the partial welding method, the key idea is to stop the welding process after we have exactly aligned the corresponding partial set of boundary points. Suppose we are given two sets of consecutive boundary points $\partial A = \{a_0,\dots,a_k,\dots,a_m\}$ and $\partial B = \{b_0,\dots,b_k,\dots,b_n\}$, where $a_i$ corresponds to $b_i$ for $i = 0,\dots,k$. This gives rise to a correspondence function $f:\gamma_{A}\subset \partial A \rightarrow \gamma_{B} \subset \partial B$, where $\gamma_{A}$ and $\gamma_{B}$ are the circular arcs formed by $\{a_0,\dots,a_k\}$ and $\{b_0,\dots,b_k\}$ respectively, such that $f(a_i) = b_i$ for $i=0,\dots,k$. Now, the objective is to find two conformal maps $\Phi_{A},\Phi_{B}$ such that $\Phi_{A}(\gamma_A) = \Phi_{B}(f(\gamma_A))$. Similar to the closed welding problem, we first find mappings $\Psi_A$ and $\Psi_B$ to map $\gamma_A$ and $\gamma_B$ to the upper and lower imaginary axis respectively, and then weld the boundary points one-by-one. The maps $\Psi_A$ and $\Psi_B$ can be realized by a half-way geodesic algorithm. The images of $\gamma_A$ and $\gamma_B$ under them are called \emph{intermediate forms}. We summarize this process in Algorithm~\ref{alg:1} as in~\cite{choi2020parallelizable}.

\begin{algorithm2e}[h!]
\KwInput{A sequence of boundary points $\{z_0,\dots,z_k,\dots,z_n\}$ constituting a closed curve and a choice of branching.}
\KwOutput{A sequence of transformed boundary points $\{Z_0,\dots,Z_k,\dots,Z_n\}$,where $Z_0,\dots,Z_k$ are on the imaginary axis according to the choice of branching.}
Let $g_1(z) = \sqrt{\dfrac{z-z_1}{z-z_0}}$ with the choice of branching; \\
\For{$j=2,\dots,k$}
{
Compute $\xi_j = (g_{j-1}\circ \cdots \circ g_{1})(z_j)$; \\
Let $g_{j}(z)=\sqrt{L_{\xi_{j}}(z)^2-1}$ with the choice of branching, where $L_{\xi_j}(z):=\dfrac{\frac{\text{Re}(\xi_j)}{\abs{\xi_j}^2}z}{1+\frac{\text{Im}(\xi_j)}{\abs{\xi_j}^2}zi}$;
}
Set $g_{k+1}(z) = \dfrac{z}{1-\frac{z}{g_k\circ g_{k-1}\circ \cdots \circ g_1(z_0)}}$;\\
Compute $Z_l = (g_{k+1}\circ \cdots \circ g_1)(z_l)$ for $l = 0,\dots,k,\dots,n$;
\caption{Intermediate form transformation} 
\label{alg:1}
\end{algorithm2e}

After performing the intermediate form transformation with respect to two different branches $(-1)^{1/2} = i$ and $(-1)^{1/2} = -i$, we obtain two set of boundary points $\{A_0,\dots,A_k$, $\dots,A_m\}$ and $\{B_0,\dots,B_k,\dots,B_n\}$, all of which are in the region $\{z \in \mathbb{C}: \text{Re}(z) \geq 0\}$. In particular, $\{A_0,\dots,A_k\}$ are on the upper imaginary axis, while the corresponding $\{B_0,\dots,B_k\}$ are on the lower imaginary axis. Next, we perform the welding step of the geodesic algorithm to weld the corresponding boundary points one-by-one conformally. The crucial point is the construction of the following M\"{o}bius transformation. Suppose $\alpha = ai$ and $\beta = bi$ are two corresponding points to be conformally aligned, where $a>0>b$. The unique M\"{o}bius transformation that maps $(\alpha,0,\beta)$ to $(i,0,-i)$ is explicitly given by
\begin{equation}
    \label{eq:14}
    T_{\alpha}^{\beta}(z) = \dfrac{z}{\frac{-2ab}{a-b}-\frac{a+b}{a-b}zi}.
\end{equation}
Consider the conformal map $z \mapsto \sqrt{z^2+1}$, which maps both $i$ and $-i$ to 0. The composition of this map and $T_{\alpha}^{\beta}$ will map both $\alpha$ and $\beta$ to 0. Now, we apply such transformations to $\{A_0,\dots,A_k\}$ and $\{B_0,\dots,B_k\}$ iteratively. In the $j$-th step, suppose we have obtained   
\begin{equation}
    \alpha_j = (h_{j-1}^{A}\circ \cdots \circ h_0^A)(A_j)
\end{equation}
and
\begin{equation}
    \beta_j = (h_{j-1}^B\circ \cdots \circ h_0^B)(B_j).
\end{equation}
Then, we define
\begin{equation}
    h_j^A(z) := \sqrt{T_{\beta_j}^{\alpha_j}(z)^2+1},\text{with branching } (-1)^{1/2} = i,
\end{equation}
and
\begin{equation}
    h_j^B(z) := \sqrt{T_{\beta_j}^{\alpha_j}(z)^2+1},\text{with branching } (-1)^{1/2} = -i.
\end{equation}
Both $h_j^A$ and $h_j^B$ are conformal as they are compositions of M\"{o}bius transformations, square maps, and square root maps, and the only difference between $h_j^A$ and $h_j^B$ is the choice of the branching. We apply these two maps to align $A_j$ and $B_j$. The images of all other points under $h_j^A$ and $h_j^B$ should also be updated in this iteration. After aligning all the corresponding points, we consider a conformal closing map $h_0$ similar to that in the geodesic algorithm:
\begin{equation}
    h_0(z) := \bigg(\dfrac{z}{1-\frac{z}{(h_1^A\circ\cdots\circ h_k^A)(\infty)}}\bigg)^2.
\end{equation}
We may also use auxiliary points $A_{m+1} = B_{n+1} = 0$ and $A_{m+2} = B_{n+2} = \infty$ to help us perform some normalization maps to obtain more regular results as proposed in~\cite{marshall2007convergence}. The detailed algorithm is summarized in Algorithm~\ref{alg:2} as in~\cite{choi2020parallelizable}.

\begin{algorithm2e}[h!]
\KwInput{Two sequences of boundary points $\{a_0,\dots,a_k,\dots,a_m\}$ and $\{b_0,\dots,b_k,\dots,b_n\}$, where $a_j$ should be aligned with $b_j$ for $j=0,\dots,k$.}
\KwOutput{Transformed data points $\{\tilde{a}_0,\dots,\tilde{a}_k,\dots,\tilde{a}_m\}$ and $\{\tilde{b}_0,\dots,\tilde{b}_k,\dots,\tilde{b}_n\}$ such that $\tilde{a}_i = \Phi_A(a_i),i=1,\dots,m$ and $\tilde{b}_i=\Phi_B(b_i),i=1,\dots,n$ for some conformal $\Phi_A$ and $\Phi_B$, and $\tilde{a}_j = \tilde{b}_j,j=0,\dots,k$.}
Define auxiliary points $a_{m+1}=b_{n+1}=0,a_{m+2}=b_{n+2}=\infty$; \\
Apply Algorithm~\ref{alg:1} on $\{a_0,\dots,a_k,\dots,a_m,a_{m+1},a_{m+2}\}$ with branching $(-1)^{1/2}=i$ to obtain $\{A_0,\dots,A_k,\dots,A_m,A_{m+1},A_{m+2}\}$. Denote the transformation by $\Psi_A$;\\
Apply Algorithm~\ref{alg:1} on $\{b_0,\dots,b_k,\dots,b_n,b_{n+1},b_{n+2}\}$ with branching $(-1)^{1/2}=-i$ to obtain $\{B_0,\dots,B_k,\dots,B_n,B_{n+1},B_{n+2}\}$. Denote the transformation by $\Psi_B$;\\
Set $h_{k-1}^{A}(z) := \sqrt{T_{B_{k-1}}^{A_{k-1}}(z)^2+1}$ with branching $(-1)^{1/2}=i$, and $h_{k-1}^{B}(z) := \sqrt{T_{B_{k-1}}^{A_{k-1}}(z)^2+1}$ with branching $(-1)^{1/2}=-i$;\\
\For{$j=k-2,\dots,1$}
{
Compute $\alpha_j = (h_{j+1}^A\circ\cdots\circ h_{k-1}^A)(A_j)$;\\
Compute $\beta_j = (h_{j+1}^B\circ\cdots\circ h_{k-1}^B)(B_j)$;\\
Set $h_j^{A}(z):=\sqrt{T_{\beta_j}^{\alpha_j}(z)^2+1}$ with branching $(-1)^{1/2}=i$ and $h_j^{B}(z):=\sqrt{T_{\beta_j}^{\alpha_j}(z)^2+1}$ with branching $(-1)^{1/2}=-i$;
}
Set $h_0(z) := \bigg(\dfrac{z}{1-\frac{z}{(h_1^A\circ\cdots\circ h_k^A)(\infty)}}\bigg)^2$;\\
Compute $\tilde{a}_l = (h_0\circ\cdots\circ h_{k-1}^A)(A_l)$ for $l=0,\dots,m+2$;\\
Compute $\tilde{b}_l = (h_0\circ\cdots\circ h_{k-1}^B)(B_l)$ for $l=0,\dots,n+2$;\\
Apply a M\"{o}bius transformation $T$ that maps $(\tilde{a}_{m+1},\tilde{b}_{n+1},\frac{1}{2}(\tilde{a}_{m+2}+\tilde{a}_{n+2}))$ to $(-1,1,\infty)$ for all points to obtain the final result.
\caption{Partial welding}
 \label{alg:2}
\end{algorithm2e}

\subsubsection{The multiply-connected case}
In the simply-connected case, when we partition the given mesh, we can ensure that the partition path is continuous. Therefore, when we apply partial welding to retrieve the entire mesh, the welding path is a continuous curve. However, this condition cannot be guaranteed for multiply-connected surfaces. On one hand, in many situations, it is natural to partition the entire mesh into several simply-connected submeshes, which could reduce the computational cost and increase the stability of the algorithm. On the other hand, imposing too many restrictions on the mesh partition step could increase the difficulty and complexity of it. As a result, dealing with situations where the welding path is discontinuous, as shown in Fig.~\ref{fig:partitions}, is inevitable.

\begin{figure}[t]
\centering
\begin{subfigure}[t]{.5\textwidth}
\centering
\includegraphics[width=.6\linewidth]{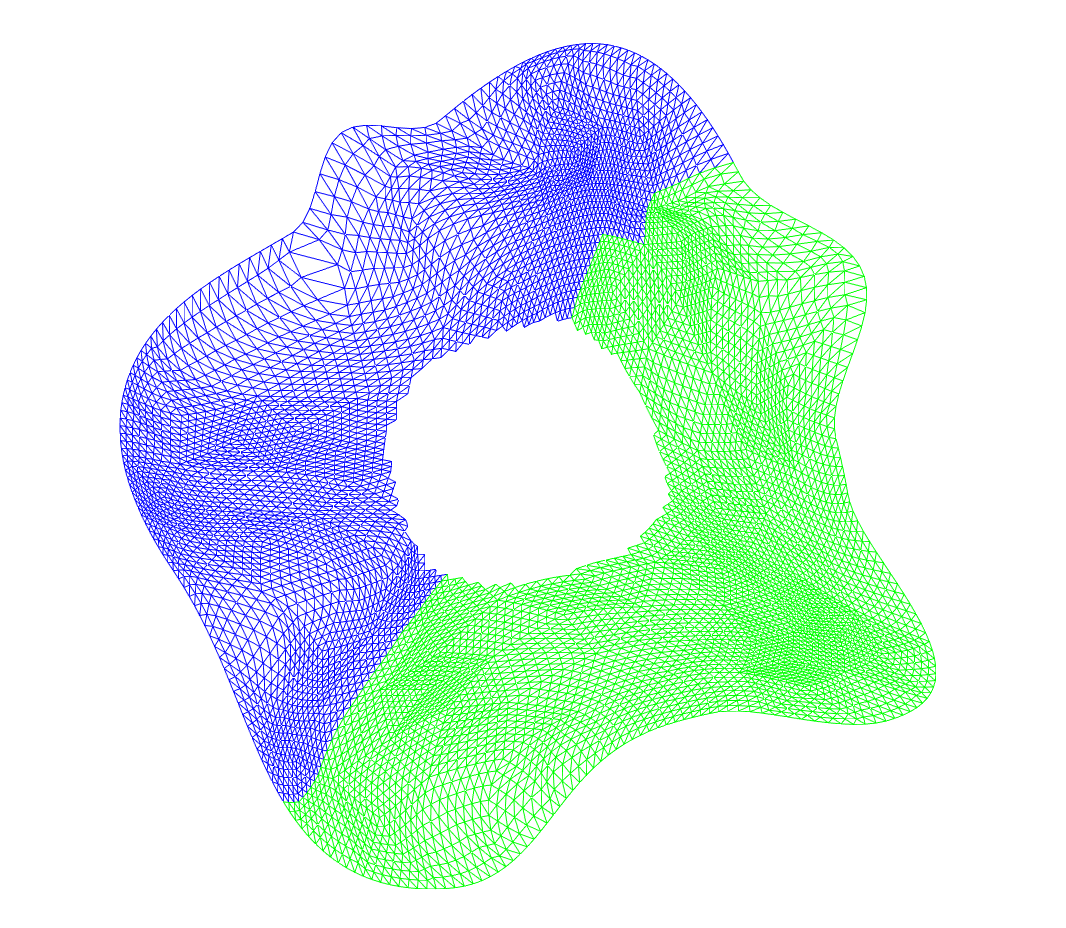}
\caption{Two submeshes}
\end{subfigure}%
\begin{subfigure}[t]{.5\textwidth}
\centering
\includegraphics[width=.6\linewidth]{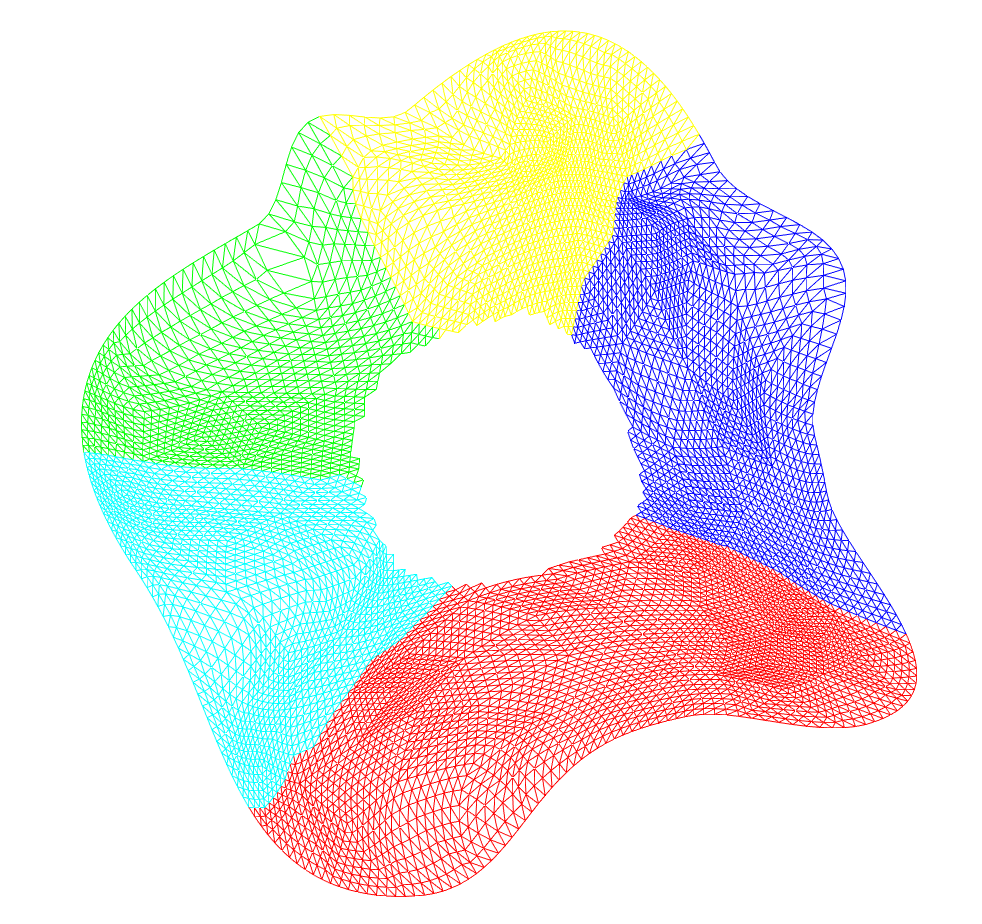}
\caption{More submeshes}
\end{subfigure}
\caption{Partitioning a multiply-connected mesh.}
\label{fig:partitions}
\end{figure}

More specifically, in Fig.~\ref{fig:partitions}(a), we partition the given mesh into two simply-connected submeshes (the blue one and the green one). It can be observed that the common boundary components of the two submeshes are two disjoint continuous arcs instead of one continuous arc. Since the inner hole is large, if we partition the surface in a way such that the inner hole is totally contained in one submesh, the welding path will contain relatively more points than the case shown in the figure, which increases the computational cost of the welding process. Also, if the inner hole is irregular in shape, imposing the requirement that it is contained in one submesh may cause the partition method to generate a highly irregular submesh, which is undesirable. Therefore, it is important to develop a welding method for handling the situation in Fig.~\ref{fig:partitions}(a). In case the partition consists of more submeshes like Fig.~\ref{fig:partitions}(b), we can weld the submeshes that share continuous boundary arcs and eventually reach the state in Fig.~\ref{fig:partitions}(a). For example, we can first weld the yellow, green, and cyan submeshes in Fig.~\ref{fig:partitions}(b) to obtain a large submesh, and weld the red and blue ones to obtain another large submesh. This simplifies the problem to the situation in Fig.~\ref{fig:partitions}(a). Besides, in case the given mesh contains multiple holes, one can further partition it so that each of the submeshes contains exactly one hole like the mesh shown in Fig.~\ref{fig:partitions}(a). Therefore, it suffices to focus on the case shown in Fig.~\ref{fig:partitions}(a) and develop a partial welding method for it. 

We now formulate the problem described above mathematically. Suppose $A,B\subset \bar{\mathbb{C}}$ are two Jordan domains with given orientations. Let $\gamma_A^1,\gamma_A^2 \subset \partial A$ be two disjoint arcs with the same orientation on $\partial A$ and $\gamma_B^1,\gamma_B^2 \subset \partial B$ be two disjoint arcs with the same orientation on $\partial B$. Suppose we are given two orientation-preserving homeomorphisms $f_1:\gamma_A^1 \rightarrow \gamma_B^1$ and $f_2:\gamma_A^2 \rightarrow \gamma_B^2$. The partial welding problem aims to find two conformal maps $\Phi_A: A\rightarrow \Omega$ and $\Phi_B: B\rightarrow \mathbb{C}\backslash \bar{\Omega}$ for some domain $\Omega$, with homeomorphic extensions to the closures, such that 
 \begin{equation}
     \Phi_A = \Phi_B\circ f_1\text{ on }\gamma_{A}^1 \ \text{ and } \Phi_A = \Phi_B\circ f_2\text{ on }\gamma_{A}^2.
 \end{equation}
 Recall that by Theorem~\ref{thm:sewing}, the closed welding problem is solvable if the given homeomorphism is quasi-symmetric on the real axis. To make use of this theorem, we extend the domain $A$ and $B$ to transform the problem to a closed welding problem. We have the following result:
 \begin{theorem}
 The above partial welding problem for multiply-connected domains can be solved by solving a closed welding problem with a suitable extension. In particular, one can extend $A$ and $B$ to two larger domains $\hat{A}$ and $\hat{B}$ and construct the maps $\Phi_A$ and $\Phi_B$ via $\hat{A}$ and $\hat{B}$. 
 \end{theorem}
 \begin{proof}
 An illustration of the construction is given in Fig.~\ref{fig:partial}. Suppose the starting and ending points of $\gamma_A^1,\gamma_A^2,\gamma_B^1,\gamma_B^2$ are $a_1^1,a_1^2,b_1^1,b_1^2$ and $a_2^1,a_2^2,b_2^1,b_2^2$ respectively. Since $\mathbb{C}\backslash A$ is multiply-connected with 1 hole, we can find a curve $\gamma_A^3 \subset \mathbb{C}\backslash A$ connecting $a_2^1$ and $a_1^2$ and a curve $\gamma_A^4 \subset \mathbb{C}\backslash A$ connecting $a_1^1$ and $a_2^2$ such that $\gamma_A^3$ is not homotopic to $\gamma_A^4$ and $\gamma_A^3 \bigcap \gamma_A^4 = \emptyset$. We then take $\hat{A}$ to be the interior of $\gamma_A = \gamma_A^1 \bigcup \gamma_A^2 \bigcup \gamma_A^3 \bigcup \gamma_A^4$. Clearly, $A\subset \hat{A}$. Similarly, we extend $B$ to a larger domain $\hat{B}$. We also extend $f_1$ and $f_2$ to a homeomorphism $f:\partial \hat{A} \to \partial \hat{B}$ such that $\restr{f}{\gamma_A^1} = f_1$ and $\restr{f}{\gamma_A^2} = f_2$.
 
 We then find two conformal maps $\psi_{\hat{A}}: \hat{A}\to \mathbb{H}$ and $\psi_{\hat{B}}: \hat{B}\to \mathbb{C}\backslash\bar{\mathbb{H}}$, which extend continuously to homeomorphisms on the boundaries. Now, the composition map $\psi_{\hat{B}}\circ f \circ \psi_{\hat{A}}^{-1}$ is a homeomorphism from $\mathbb{R}$ to $\mathbb{R}$. By Theorem~\ref{thm:sewing}, if $\psi_{\hat{B}}\circ f \circ \psi_{\hat{A}}^{-1}$ is quasi-symmetric, we can find conformal maps $\phi_{\hat{A}}:\mathbb{H}\to \Omega$ and $\phi_{\hat{B}}:\mathbb{C}\backslash \bar{\mathbb{H}}\to \mathbb{C}\backslash \bar{\Omega}$ for some Jordan domain $\Omega$ such that $\phi_{\hat{A}}=\phi_{\hat{B}}\circ f$ on $\mathbb{R}$. Finally, we take $\Phi_A = \restr{\psi_{\hat{A}}\circ \phi_{\hat{A}}}{A}$ and $\Phi_B = \restr{\psi_{\hat{B}}\circ \phi_{\hat{B}}}{B}$. It is easy to see that $\Phi_A$ and $\Phi_B$ give the desired partial welding maps.
 \end{proof}
 
 \begin{figure}[t]
\centering
  \includegraphics[width=0.9\linewidth]{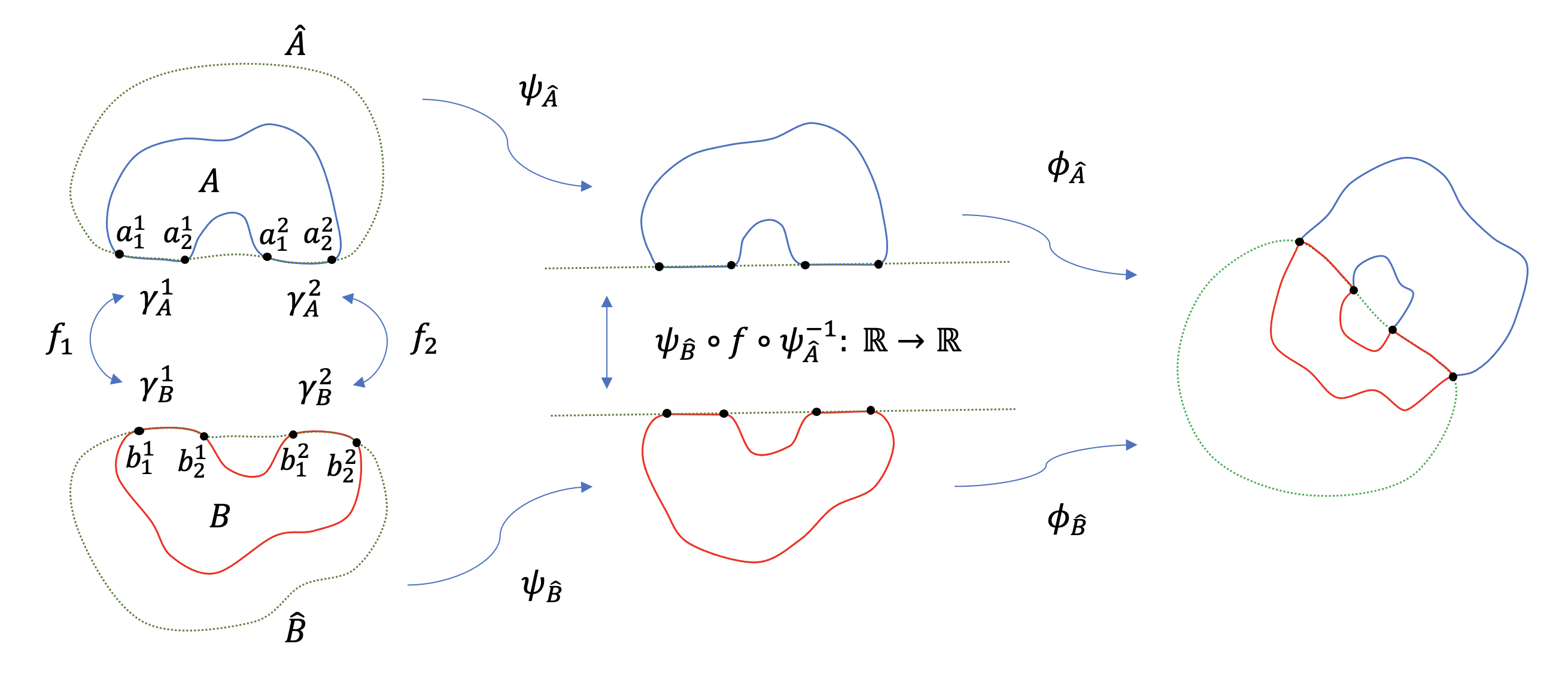}
\caption{The theoretical construction for partial welding for multiply-connected domains.}
\label{fig:partial}
\end{figure} 

In the discrete case, suppose we are given two set of consecutive boundary points $\partial A = \{a_0,\dots,a_r,\dots,a_s$, $\dots,a_t,\dots,a_m\}$ and $\partial B = \{b_0,\dots,b_r,\dots,b_s,\dots,b_t,\dots,b_n\}$, where $a_i$ corresponds to $b_i$ for $i = 1,\dots,r,s,\dots,t$, and $\{a_r,\dots,a_s\}$ and $\{b_r,\dots,b_s\}$ correspond to the inner boundary of the original mesh. This gives the correspondence functions $f_1:\gamma_A^1\rightarrow \gamma_B^1$ and $f_2:\gamma_A^2\rightarrow \gamma_B^2$, where $\gamma_A^1,\gamma_A^2,\gamma_B^1,\gamma_B^2$ are formed by  $\{a_0,\dots,a_r\},\{a_s,\dots,a_t\},\{b_0,\dots,b_r\},\{b_s,\dots,b_t\}$ respectively. Let $A$ and $B$ denote the polygons enclosed by $\partial A = \{a_0,\dots,a_r,\dots,a_s,\dots,a_t$, $\dots,a_m\}$ and $\partial B = \{b_0,\dots,b_r,\dots,b_s,\dots,b_t,\dots,b_n\}$ respectively. Our goal is to find conformal maps $\Phi_A$ and $\Phi_B$ such that $\Phi_A(\gamma_A^1) = (\Phi_B\circ f_1)(\gamma_A^1)$ and $\Phi_A(\gamma_A^2) = (\Phi_B\circ f_2)(\gamma_A^2)$. To compute the partial welding maps, we follow the idea of the theoretical construction. More specifically, we find auxiliary points $\{\bar{a}_1,\dots,\bar{a}_k\}$ and $\{\bar{b}_1,\dots,\bar{b}_k\}$ such that none of $\{\bar{a}_1,\dots,\bar{a}_k\}$ are contained in the polygon $A$, none of $\{\bar{b}_1,\dots,\bar{b}_k\}$ are contained in the polygon $B$, and $\{a_0,\dots,a_r,\bar{a}_1,\dots,\bar{a}_k,a_s,\dots,a_t$, $\dots,a_m\}$ and $\{b_0,\dots,b_r,\bar{b}_1,\dots,\bar{b}_k,b_s,\dots,b_t,\dots,b_n\}$ form two larger polygons with the length of each edge sufficiently small, respectively. We require that the length of edges of the new polygon are sufficiently small because it ensures a good approximation of the desired conformal map computed by the geodesic algorithm as described in~\cite{marshall2007convergence}. We then compute the desired partial welding maps $\Phi_A$ and $\Phi_B$ with the path correspondence between  $\{a_0,\dots,a_r,\bar{a}_1,\dots,\bar{a}_k,a_s,\dots,a_t\}$ and $\{b_0,\dots,b_r,\bar{b}_1,\dots,\bar{b}_k,b_s,\dots,b_t\}$. After that, we discard the polygon enclosed by $\Phi_A(\bar{a}_1),\dots,\Phi_A(\bar{a}_k)$, $\Phi_B(\bar{b}_1),\Phi_B(\bar{b}_k)\}$ to obtain the desired multiply-connected mesh. 

Note that there are various ways to find the auxiliary points. In most cases, we can choose them to be points on the straight lines between $a_r$ and $a_s$ and between $b_r$ and $b_s$, i.e.,
\begin{equation}
    \bar{a}_i = \frac{i}{k+1}a_r + (1-\frac{i}{k+1})a_s,
\end{equation}
and
\begin{equation}
    \bar{b}_i = \frac{i}{k+1}b_r + (1-\dfrac{i}{k+1})b_s,
\end{equation}
for $i = 1,\dots,k$. Another possible choice is the circular arc connecting $a_{r-1},a_r$, and $a_s$. Note that the straight lines between $a_r$ and $a_s$ and between $b_r$ and $b_s$ generally work well for the partial welding method. More specifically, suppose the original mesh is partitioned into two submeshes as shown in Fig.~\ref{fig:welding_multi}(a). The line connecting the starting and ending points $a_r$ and $a_s$ is in the inner hole of the mesh in most cases. As conformal maps and quasi-conformal maps with small $|\mu|$ tend to preserve the local geometry of the mesh, the line connecting $a_r$ and $a_s$ should lie outside the transformed submeshes if the distortion is small enough as shown in Fig.~\ref{fig:welding_multi}(b)--(c). The partial welding method can then be applied to weld the two submeshes as shown in Fig.~\ref{fig:welding_multi}(d). For some extreme cases where the straight lines do not lie outside the submeshes, we may apply some other path-finding algorithms such as~\cite{garcia1997path,yap2002grid} for getting the auxiliary points. The proposed partial welding method is summarized in Algorithm~\ref{alg:3}.

\begin{figure}[t]
\centering
\begin{subfigure}[t]{.48\textwidth}
\centering
\includegraphics[width=.45\linewidth]{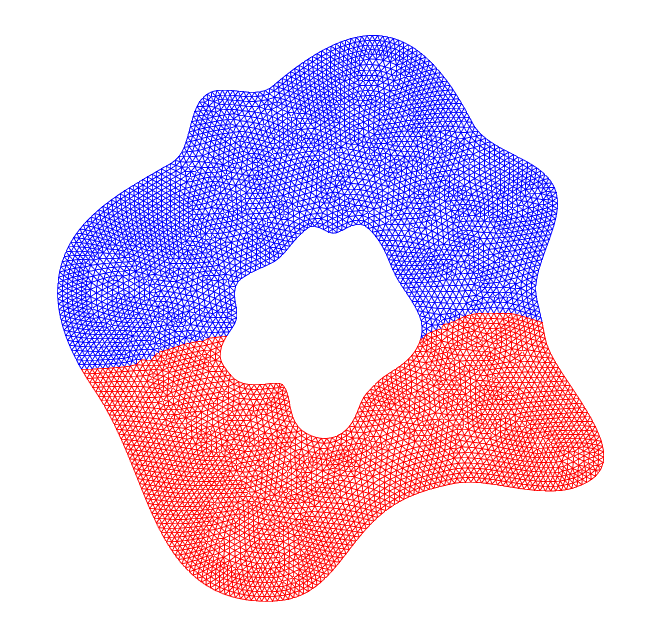}
\caption{The given mesh partitioned into 2 submeshes}
\end{subfigure}%
\hfill
\begin{subfigure}[t]{.48\textwidth}
\centering
\includegraphics[width=.45\linewidth]{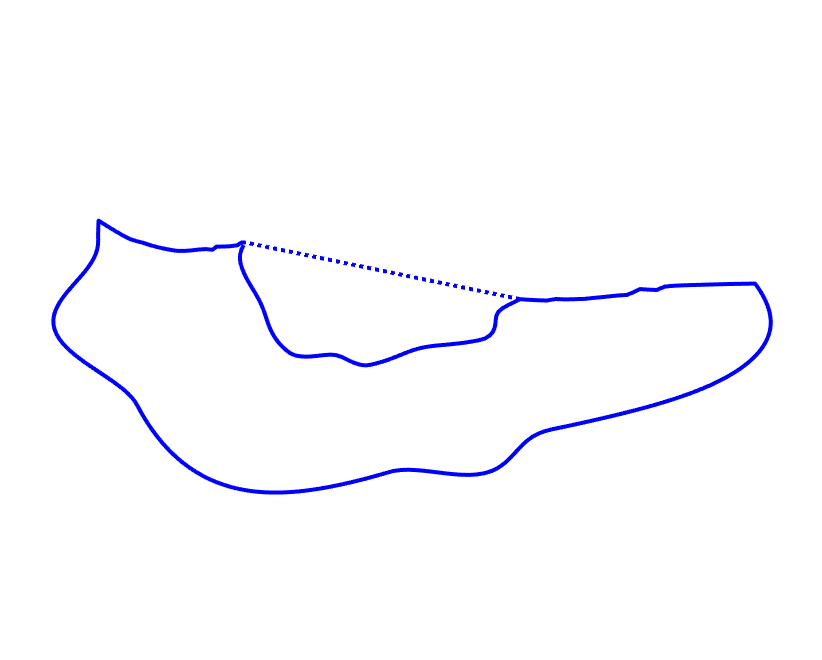}
\caption{Transformed blue submesh with auxiliary path}
\end{subfigure}
\begin{subfigure}[t]{.48\textwidth}
\centering
\includegraphics[width=.45\linewidth]{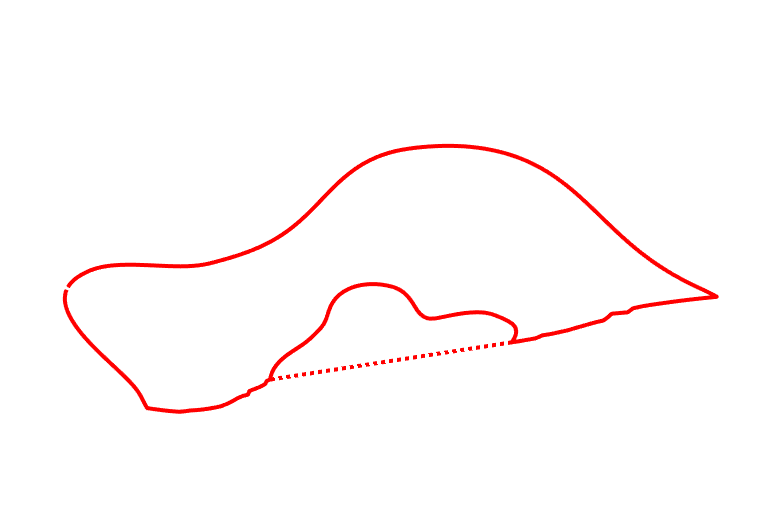}
\caption{Transformed red submesh with auxiliary path}
\end{subfigure}%
\hfill
\begin{subfigure}[t]{.48\textwidth}
\centering
\includegraphics[width=.45\linewidth]{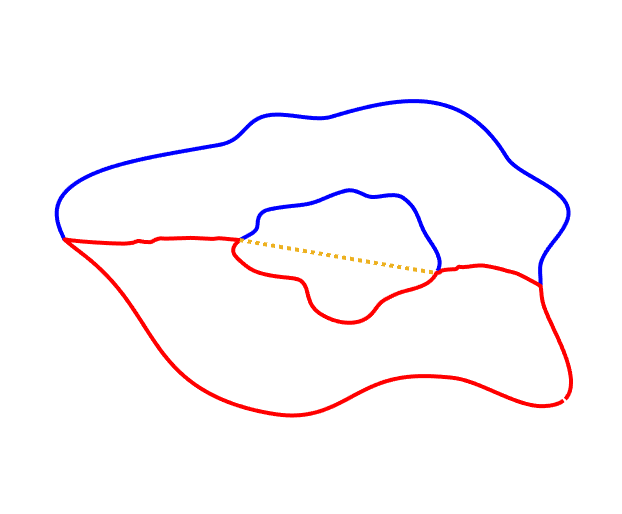}
\caption{The welded mesh}
\end{subfigure}
\caption{An illustration of the proposed partial welding method for multiply-connected meshes.}
\label{fig:welding_multi}
\end{figure}

\begin{algorithm2e}[h!]
\KwInput{Two sequences of boundary points $\partial A = \{a_0,\dots,a_r,\dots,a_s,\dots, a_t,\dots,a_m\}$ and $\partial B = \{b_0,\dots,b_r,\dots,b_s,\dots,b_t,\dots,b_n\}$, where $a_i$ are to be aligned with $b_i$ for $i=0,\dots,r,s,\dots,t$, and $a_r,\dots,a_s,b_r,\dots,b_s$ are taken from the inner boundaries.}
\KwOutput{Conformally transformed points $\{\tilde{A_0},\dots,\tilde{A}_r,\dots,\tilde{A}_s,\dots,\tilde{A}_t,\dots,\tilde{A}_m\}$ and $\{\tilde{B_0},\dots,\tilde{B}_r,\dots,\tilde{B}_s,\dots,\tilde{B}_t,\dots,\tilde{B}_m\}$ such that $\tilde{A}_i = \tilde{B}_i$ for $i=0,\dots,r,s$, $\dots,t$, and the transformed points form a multiply-connected polygon.}
Find auxiliary points $\bar{a}_1,\dots,\bar{a}_k$ and $\bar{b}_1,\dots,\bar{b}_k$ such that they are not in the polygons $A$ and $B$ respectively. Also, $\{a_0,\dots,a_r,\bar{a}_1,\dots,\bar{a}_k,a_s,\dots,a_t,\dots,a_m\}$ and $\{b_0,\dots,b_r,\bar{b}_1,\dots,\bar{b}_k,b_s,\dots,b_t,\dots,b_n\}$ form two larger polygons; \\
Apply the partial welding algorithm (Algorithm~\ref{alg:2}) with path correspondence $\{a_0,\dots$, $a_r,\bar{a}_1,\dots,\bar{a}_k,a_s,\dots,a_t\}$ and $\{b_0,\dots,b_r,\bar{b}_1,\dots,\bar{b}_k,b_s,\dots,b_t\}$ to update $\partial A $ and $\partial B$;\\
The new coordinates of $\partial A$ and $\partial B$ give the desired map.
\caption{Partial welding for multiply-connected meshes} 
\label{alg:3}
\end{algorithm2e}

\begin{figure}[t]
\centering
\begin{subfigure}[t]{.5\textwidth}
\centering
\includegraphics[width=.5\linewidth]{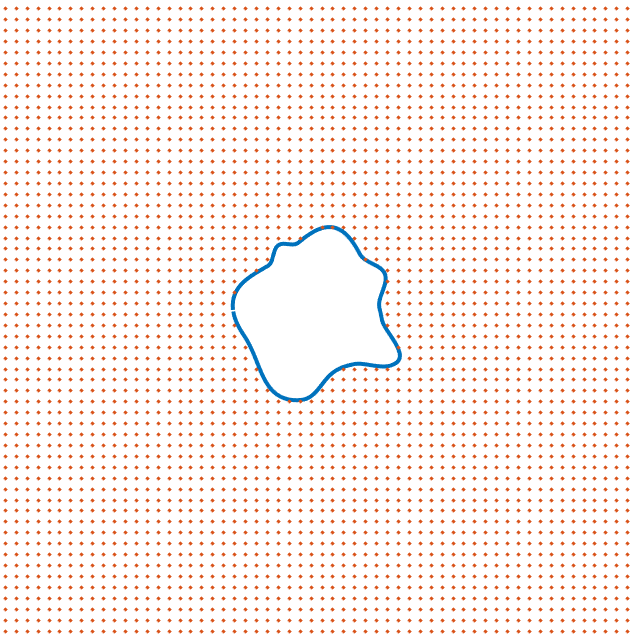}
\caption{The original curve and its exterior}
\end{subfigure}%
\begin{subfigure}[t]{.5\textwidth}
\centering
\includegraphics[width=.5\linewidth]{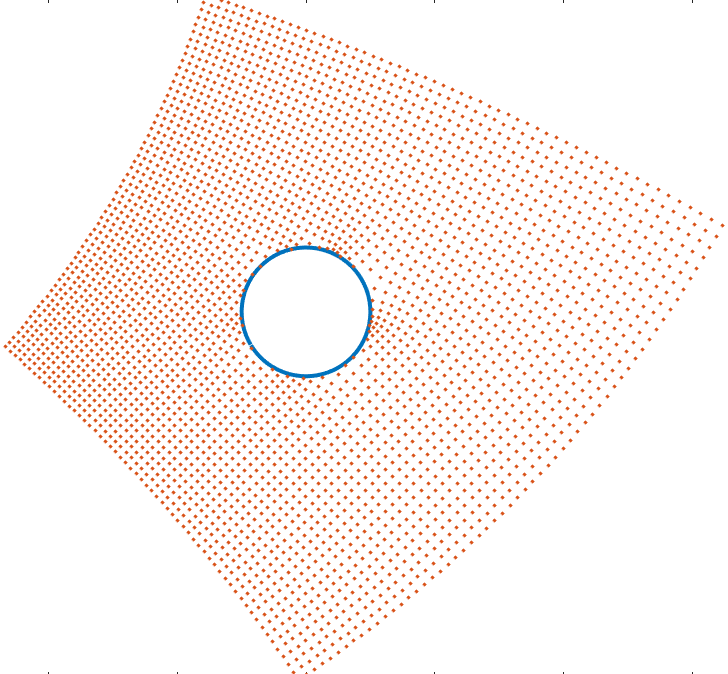}
\caption{The computed conformal map}
\end{subfigure}
\begin{subfigure}[t]{.5\textwidth}
\centering
\includegraphics[width=.5\linewidth]{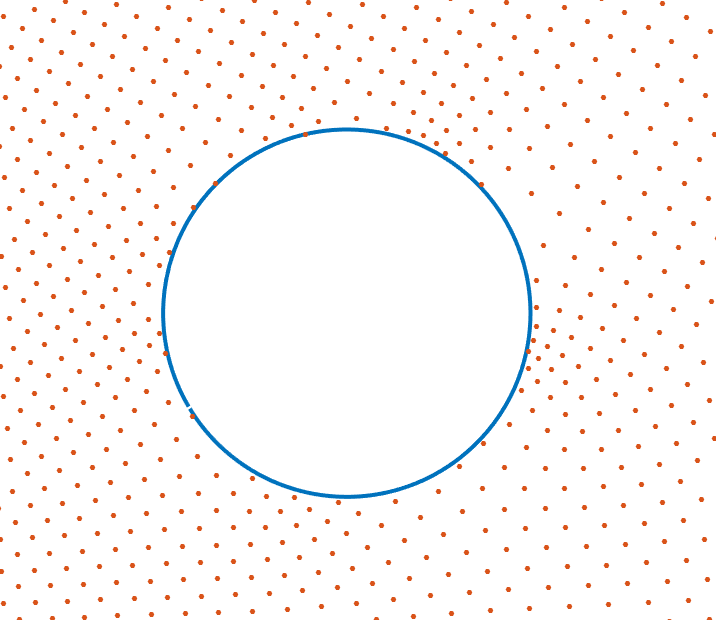}
\caption{Zoom-in of the region near the circle}
\end{subfigure}%
\begin{subfigure}[t]{.5\textwidth}
\centering
\includegraphics[width=.5\linewidth]{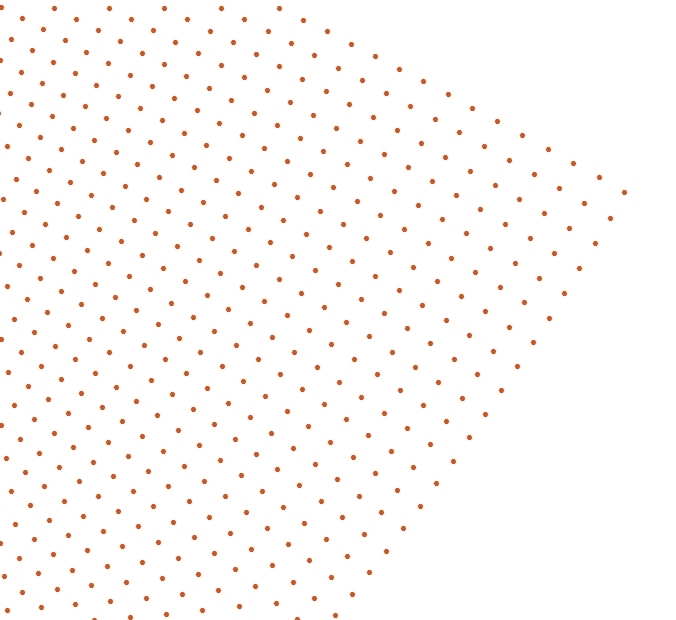}
\caption{The map causes a rotation far away from the circle}
\end{subfigure}%
\caption{The conformal map from the exterior of a curve to the exterior of a circle computed using the geodesic algorithm.}
\label{fig:normalization}
\end{figure}

\subsection{Parallel Koebe's iteration}
\label{sec:4.5}
As introduced in Section~\ref{sec:3.5}, when we perform the Koebe's iteration for a domain $R$ whose complements are $K_1,K_2,\dots,K_n$, in each iteration, we normalize the iteration map $f_j$ at $\infty$ such that $f_j(z) = z + O(\frac{1}{z})$. This plays an essential role for ensuring the convergence of the Koebe's iteration. Intuitively, with $f_j$ normalized at $\infty$, it only changes the region near $K_j$ while being close to the identity map (possibly with a rotation) locally for points far away from $K_j$. Moreover, suppose in the $(j-1)$-th iteration we have transformed the inner boundary of $K_{j-1}$ to a circle. Then, in the $j$-th iteration, the transformed inner boundary will still be similar to a circle if $K_{j-1}$ is far away from $K_j$. Computationally, the normalization step is incorporated as the last step of the geodesic algorithm as in~\cite{marshall2007convergence}. Fig.~\ref{fig:normalization} shows an example of the effect, from which it can be observed that the region near the curve is significantly changed under the map while the region far away from it is only rotated but not distorted locally. This motivates us to design a parallelizable version of the Koebe's iteration method for our parallel quasi-conformal parameterization problem.

\begin{figure}[t]
\centering
\begin{subfigure}[t]{.5\textwidth}
\centering
\includegraphics[width=.6\linewidth]{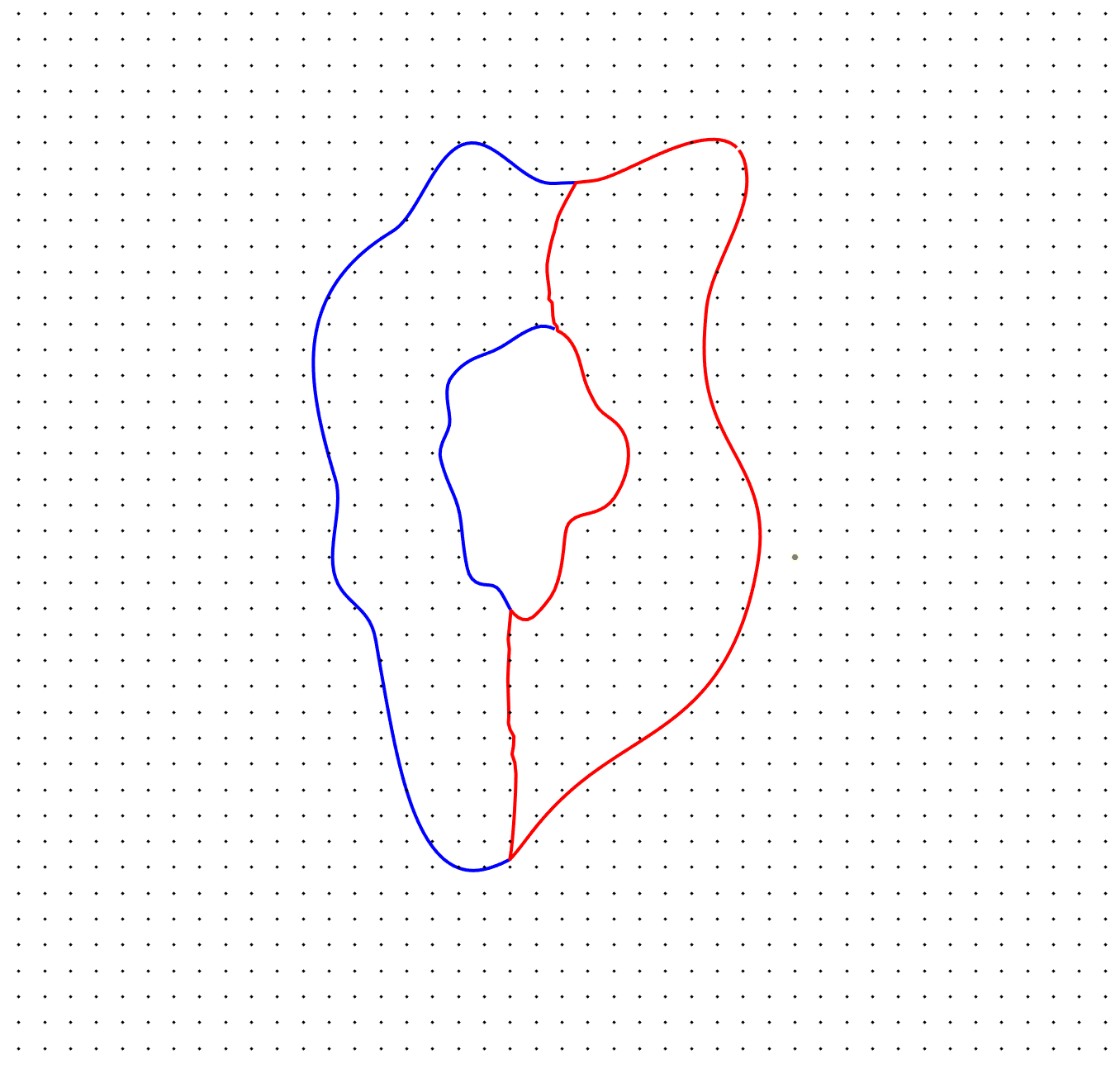}
\caption{Before the transformation}
\end{subfigure}%
\begin{subfigure}[t]{.5\textwidth}
\centering
\includegraphics[width=.6\linewidth]{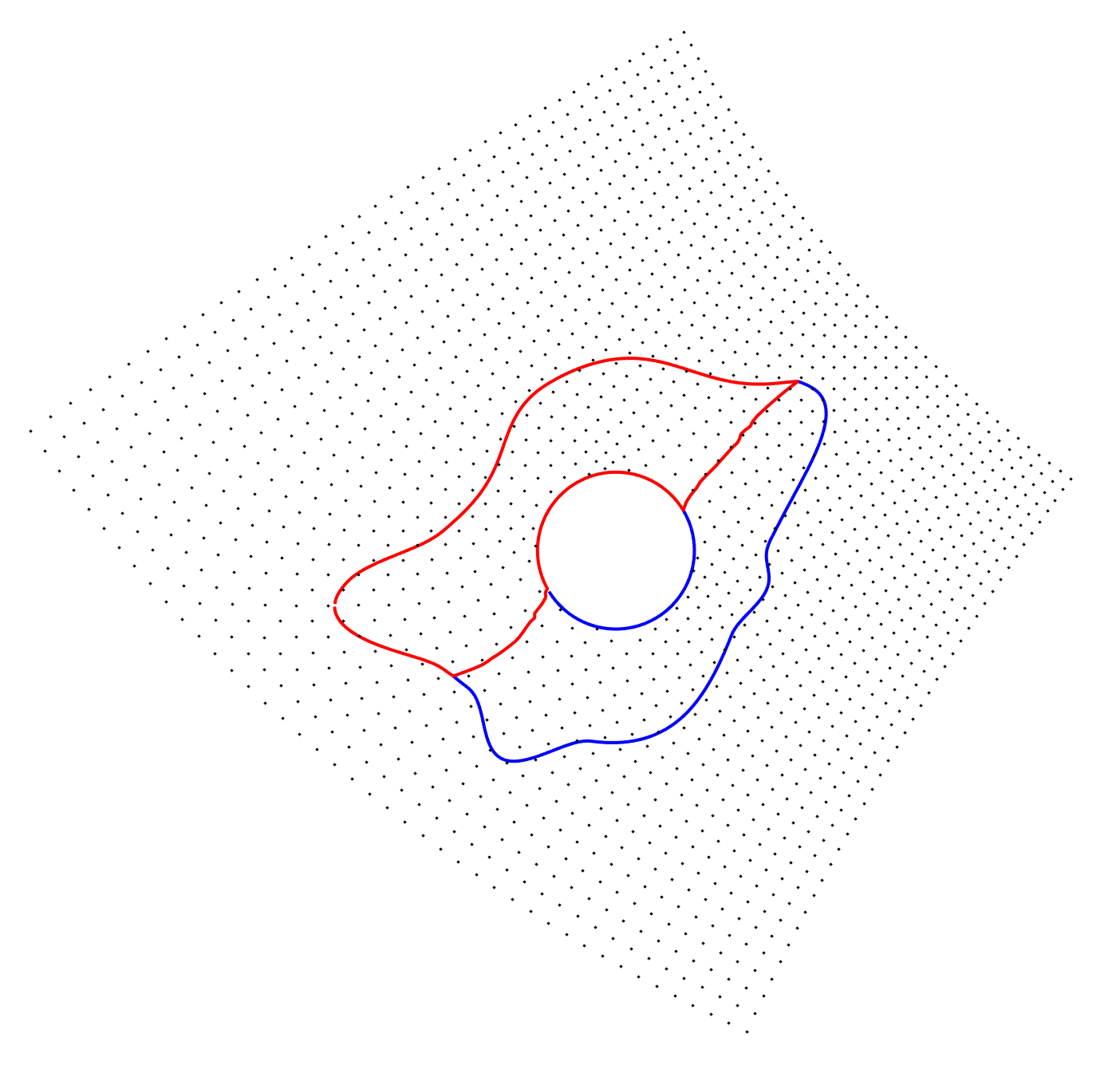}
\caption{After the transformation}
\end{subfigure}%
\caption{Transforming the inner boundary of a one-hole submesh $\mathcal{S}_j$ to a circle using the geodesic algorithm under normalization.}
\label{fig:koebe_1hole}
\end{figure}

Let $\mathcal{S}$ be the input multiply-connected mesh with exactly $k$ inner holes. Suppose $\mathcal{S}$ is partitioned into $m$ submeshes $\mathcal{S}_1, \dots, \mathcal{S}_m$, where each of them is either simply-connected or multiply-connected with one inner hole, and the free-boundary quasi-conformal parameterizations of them obtained using Algorithm~\ref{alg:free} are $\varphi_1, \dots, \varphi_m$. Instead of merging all flattened submeshes and performing the traditional Koebe's iteration on the entire mesh directly, we apply the geodesic algorithm to weld some of the submeshes so that each of the $k$ holes is contained in one of $\varphi_j(\mathcal{S}_j)$ or a welded larger subdomain. We then transform the inner boundary $\mathcal{H}_j$ of each subdomain $\varphi_j(\mathcal{S}_j)$ (or a welded larger subdomain) to a circle in parallel. More explicitly, we find a normalized conformal map $\Phi_j: \mathbb{C}\backslash \mathcal{H}_j \to \mathbb{C}\backslash B(0,1)$ satisfying $\Phi_j(\partial\mathcal{H}_j) = S(0,1), \Phi_j(\infty) = \infty$ and $\Phi_j(a_0) = 0$ for some point $a_0$ on $\partial\mathcal{H}_j$, where $B(0,1)$ denotes the unit ball and $S(0,1)$ denotes the unit circle. All the boundary points and welding paths related to $\varphi_j(\mathcal{S}_j)$ should be updated. Fig.~\ref{fig:koebe_1hole} shows an example of the transformation. Since the transformations of all $\mathcal{H}_j$ to circles are independent, in practice they can be computed by different processors in a parallel manner. After computing all transformations, we obtain the updated boundaries of the submeshes $\tilde{\mathcal{S}}_1,\dots, \tilde{\mathcal{S}}_k, \dots$, which are either simply-connected or multiply-connected with 1 circular hole. We can then perform the remaining welding steps for getting the entire boundaries. Fig.~\ref{fig:parallelkoebe} shows an example of the computation, in which we handle the two submeshes with 1 hole in (a) and (b) in parallel to get the results in (c) and (d), and then weld them to get the result in (e). Note that the normalization step of partial welding tends to preserve the circular shapes of all inner boundaries. As for the outer boundary of the entire mesh, we can apply the geodesic algorithm to transform it to a circle after all the welding steps. This completes our parallel Koebe's iteration method. We remark that the interior of the submeshes does not need to be updated throughout the process as we will only utilize the boundary points of them for computing the desired parameterization later.

\begin{figure}[t]
\centering
\begin{subfigure}[t]{.48\textwidth}
\centering
\includegraphics[width=.5\linewidth]{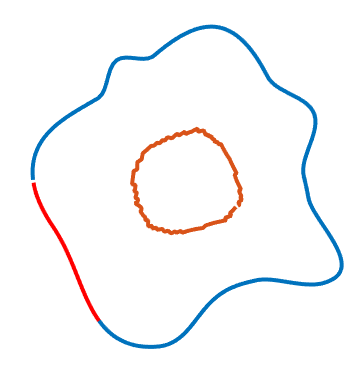}
\caption{The first multiply-connected submesh}
\label{fig:parallelkoebe1}
\end{subfigure}%
\hfill   
\begin{subfigure}[t]{.48\textwidth}
\centering
\includegraphics[width=.5\linewidth]{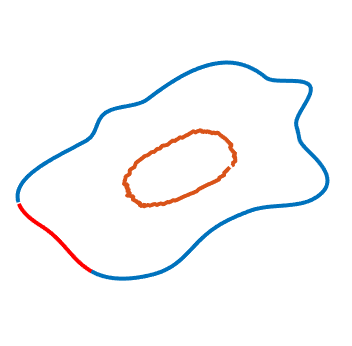}
\caption{The second multiply-connected submesh}
\end{subfigure}
\begin{subfigure}[t]{.48\textwidth}
\centering
\includegraphics[width=.5\linewidth]{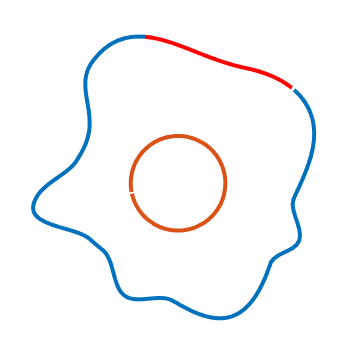}
\caption{Transforming the inner boundary of (a) to a circle}
\label{fig:parallelkoebe3}
\end{subfigure}
\hfill   
\begin{subfigure}[t]{.48\textwidth}
\centering
\includegraphics[width=.5\linewidth]{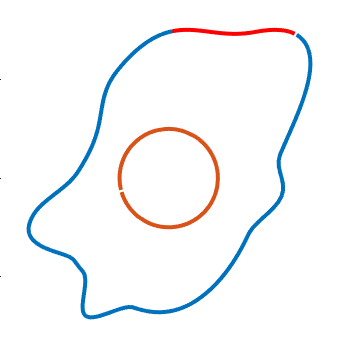}
\caption{Transforming the inner boundary of (b) to a circle}
\label{fig:parallelkoebe4}
\end{subfigure} \\
\begin{subfigure}[t]{.7\textwidth}
\centering
\includegraphics[width=.7\linewidth]{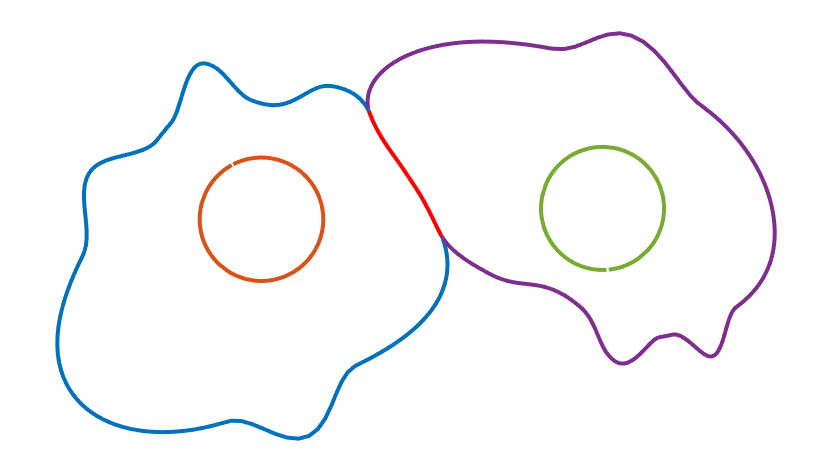}
\caption{The inner boundaries remain close to circles after partial welding}
\label{fig:parallelkoebe5}
\end{subfigure}%
\caption{An illustration of the parallel Koebe's iteration.}
\label{fig:parallelkoebe}
\end{figure}

To see the advantage of the proposed parallel Koebe's iteration, note that the complexity of the geodesic algorithm is $O(n_b n_t)$, where $n_b$ is the number of boundary points that determine the map and $n_t$ is the total number of points we want to update using the map. We now consider the computational cost of the traditional (non-parallel) Koebe's iteration method and our proposed parallel Koebe's iteration. To transform the inner boundary of each subdomain with 1 hole to a circle using the geodesic algorithm, the number $n_b$ is the number of boundary points of such subdomain in both versions of the Koebe's iteration. As for the number $n_t$, if we perform the traditional Koebe's iteration on the entire welded shape, $n_t$ will be the total number of boundary points of the submeshes. By contrast, in the parallel Koebe's iteration, we only need to update the new coordinates of the boundary points of each submesh in parallel, hence $n_t$ is just the number of boundary points of each submesh for computing each map. Therefore, the computational cost can be greatly reduced in the parallel Koebe's iteration by utilizing multiple processors, especially if the mesh $\mathcal{S}$ has many inner boundaries.

We remark that the parallel Koebe's iteration method is developed based on the observation that the transformed circles remain close to circles under the subsequent welding maps, and hence the result obtained by the algorithm is only an approximation of the desired Riemann mapping in theory. To ensure that all boundaries are perfectly circular, one needs to repeat the Koebe's iteration for infinitely many times so that the result will converge to the desired Riemann mapping as guaranteed by Theorem~\ref{thm:Koebe}. Nevertheless, in practice we find that the results produced by the proposed parallel Koebe's iteration algorithm without repeating the iterations are already satisfactory, with all holes being very close to perfect circles. In case the precision of the circularity of the holes is required to be particularly high, one can repeat the iterations for several times to further improve the circularity.

\subsection{Obtaining the global quasi-conformal parameterization}
After getting the updated boundary conditions for all submeshes from the above procedures, we compute the desired free-boundary quasi-conformal parameterization of each of them. Suppose the boundary of each initially flattened submesh $\varphi_i(\mathcal{S}_i)$ is $\mathcal{B}_i$, where $i=1,\dots,m$. After the steps of partial welding and parallel Koebe's iteration, we obtain the updated boundary $\tilde{\mathcal{B}}_i$ for each submesh. Now, we use the updated boundary to obtain the desired quasi-conformal parameterization for each submesh. Since a conformal map is harmonic, we can solve the Laplace equation for each flattened subdomain $\varphi_i(\mathcal{S}_i)$ to obtain a conformal map $\Phi_i: \varphi_i(\mathcal{S}_i) \to \mathbb{R}^2$:
\begin{equation}
    \Delta \Phi_i = 0,\ \Phi_i|_{\mathcal{B}_i} = \tilde{\mathcal{B}}_i.
\end{equation}
Note that since the computation for each submesh is independent, this step is highly parallelizable. We can further reduce the quasi-conformal distortion of each $\Phi_i$ by composing $\Phi_i$ with a quasi-conformal map with Beltrami coefficient computed using the composition formula \ref{eq:3} as suggested by~\cite{choi2015fast}. Note that since $\Phi_i$ is conformal, the Beltrami coefficient of the composition map $\Phi_i \circ \varphi_i:\mathcal{S}_i \to \mathbb{R}^2$ is the same as that of $\varphi_i$, which is obtained based on the input Beltrami coefficient $\mu$. In other words, this step of solving the Laplace equation for each submesh ensures the consistency of the boundaries of all submeshes without affecting their quasi-conformality. Finally, all mapping results $\Phi_i \circ \varphi_i$, $i = 1, \dots, m$ together form the desired global quasi-conformal parameterization $\Phi:\mathcal{S} \to \mathbb{R}^2$, with the Beltrami coefficient of $\Phi$ being $\mu$ and all $k$ holes of $\Phi(\mathcal{S})$ very close to circles. Our parallelizable global quasi-conformal mapping (PGQCM) method for multiply-connected surfaces is summarized in Algorithm~\ref{alg:4}. 

\begin{algorithm2e}[h!]
\KwInput{A multiply-connected surface mesh $\mathcal{S}=(\mathcal{V},\mathcal{F})$ with $k$ inner holes, and a prescribed Beltrami coefficient $\mu$, and a partition of $\mathcal{S}$ into $m$ submeshes $\mathcal{S}_i = (\mathcal{V}_i,\mathcal{F}_i)$, $i = 1, \dots, m$.}
\KwOutput{A global quasi-conformal parameterization $\Phi:\mathcal{S}\to \mathbb{R}^2$.}

\For{$i=1,\dots,m$}
{
Apply Algorithm~\ref{alg:free} to obtain the free-boundary parameterization $\varphi_i: \mathcal{S}_i \to \mathbb{R}^2$\;
}
Perform partial welding on $\varphi_i(\mathcal{S}_i)$ using Algorithm~\ref{alg:3} to get $\tilde{\mathcal{S}}_j, j=1,\dots,k,\dots$, such that each of them is simply-connected or with only one hole and each of the $k$ holes is contained in one of $\tilde{\mathcal{S}}_j$;\\
Apply the geodesic algorithm to transform the inner boundaries of all 1-hole submeshes to circles;\\
Perform partial welding to ensure the consistency of all boundaries of the submeshes\;
Apply the geodesic algorithm to transform the outer boundary to a circle\;
(Optional) Further perform the Koebe's iteration to improve the circularity of the inner holes\;
\For{$i=1,\dots,m$}
{
Solve the Laplace equation $\Delta {\Phi}_i = 0$ with the updated boundary conditions for $\varphi_i(\mathcal{S}_i)$\;
(Optional) Compose the map with a quasi-conformal map to further reduce the quasi-conformal distortion\;
}
The maps $\Phi_i \circ \varphi_i: \mathcal{S}_i \to \mathbb{R}^2$, $i = 1, \dots, m$ together form the desired map $\Phi:\mathcal{S}\to \mathbb{R}^2$\;
\caption{Parallelizable global quasi-conformal mapping for multiply-connected surfaces (PGQCM)}
\label{alg:4}
\end{algorithm2e}

The convergence of the PGQCM algorithm is guaranteed by the following theorem:
\begin{theorem}
Let $\mathcal{S}$ be a multiply-connected open surface and $\mu$ be a prescribed Beltrami coefficient. If the Koebe's iteration is repeated for infinitely many times, the map $\Phi:\mathcal{S}\to \mathbb{R}^2$ obtained by the PGQCM algorithm converges to the quasi-conformal Riemann mapping in Theorem~\ref{thm:qcm}, where all boundaries of $\Phi(\mathcal{S})$ circular and the Beltrami coefficient of $\Phi$ is $\mu$.
\end{theorem}
\begin{proof}
Note that the circularity of the holes is ensured by the convergence of the Koebe's iteration in Theorem~\ref{thm:Koebe}. Also, the error in the Beltrami coefficient of the map can be corrected by composing a quasi-conformal map using the composition formula as introduced in~\cite{choi2015fast} and hence one can ensure that the Beltrami coefficient of $\Phi$ is equal to the prescribed $\mu$.
\end{proof}

Altogether, the novel combination of the free-boundary local parameterization of the submeshes, the partial welding method, and the parallel Koebe's iteration significantly improves the computational efficiency of the global quasi-conformal parameterization of multiply-connected surfaces. Furthermore, for some very dense meshes, traditional global parameterization methods may fail due to insufficient memory size of the computing machines for solving extremely large systems of equations. By contrast, our proposed PGQCM algorithm can effectively handle any dense mesh as it does not require solving any equations for the global mesh. Instead, we can partition the input mesh into multiple submeshes such that each of them is small enough for the computing machine to compute the free-boundary quasi-conformal parameterization. After that, we can perform the partial welding and the parallel Koebe's iteration to weld and update the boundaries, and finally solve the Laplace equation for each submesh with the updated boundary conditions to get the desired mapping for each of them, thereby yielding the desired global quasi-conformal parameterization of the dense mesh. Note that the accuracy of the welding maps is theoretically ensured as described in~\cite{marshall2007convergence}. In the experiments presented in the following section, one can see that the proposed PGQCM method is not only more efficient but also more accurate than the existing methods, especially for dense meshes.

\section{Experiments}
\label{sec:5}
Our proposed algorithm is implemented in MATLAB, with the MATLAB Parallel Computing Toolbox used for performing the parallel computation in our algorithm. The sparse linear systems are solved using the backslash operator ($\backslash$) in MATLAB. The numerical calculations are done using the default 16-digit precision in MATLAB. All experiments are performed on a MacBook Pro with 2.3 GHz 8-Core Intel Core i9 CPU and 16 GB RAM. Various synthetic and real multiply-connected mesh models from~\cite{RiemannMapper,choi2021efficient} are used for assessing the performance of our proposed algorithm.

\subsection{Error estimate}
For a given multiply-connected surface $\mathcal{S} = (\mathcal{V},\mathcal{F})$ and a prescribed piecewise constant Beltrami coefficient $\mu$ defined on each face of $\mathcal{S}$, let $\Phi:\mathcal{S} \to \mathbb{C}$ be the computed quasi-conformal parameterization. We first measure the error $e_T$ between the Beltrami coefficient $\mu_{\Phi}$ of the resulting map $\Phi$ and the ground truth Beltrami coefficient $\mu$ on each face $T \in \mathcal{F}$:
\begin{equation}
    e_T = (\mu_{\Phi} - \mu)|_{T},
\end{equation}
where $\mu_{\Phi}|_T$ is the Beltrami coefficient of the linear map from $T = [v_i,v_j,v_k]$ to $\Phi(T) = [\Phi(v_i),\Phi(v_j),\Phi(v_k)]$. We can then compute the mean absolute error
\begin{equation}
    e = \underset{T\in\mathcal{F}}{\text{mean}}\; \abs{e_T}.
\end{equation}
Note that we do not adopt the relative error $\frac{|e_T|}{|\mu|_T|}$ or $\frac{|e_T|}{\text{mean}(|\mu|)}$ here because if the ground truth $\mu$ is identically zero, i.e., the desired parameterization is conformal, then the relative error will be $\infty$ no matter how small $e_T$ is and hence is not a good measure. Since the Beltrami coefficient effectively captures the local geometric distortion of the parameterization, a small mean absolute error $e$ indicates that the conformality distortion between the desired parameterization and the computed parameterization is very small. Mathematically, this can be seen from the composition formula of Beltrami coefficients in Equation~\eqref{eq:3}. If we denote the ground truth quasi-conformal map as $\Psi:\mathcal{S}\to \mathbb{C}$, then we have
\begin{equation}
    \mu_{\Psi^{-1}}(\Psi(z_0)) = -\mu_{\Psi}(z_0)\dfrac{\Psi_{z}(z_0)}{\bar{\Psi}_{z}(z_0)},
\end{equation}
and
\begin{equation}
    \mu_{\Phi\circ\Psi^{-1}}(\Psi(z_0))
    = \dfrac{\mu_{\Psi^{-1}}+(\mu_{\Phi}\circ\Psi^{-1})\frac{\bar{\Psi}^{-1}_{z}}{\Psi^{-1}_{z}}}{1+\bar{\mu}_{\Psi^{-1}}(\mu_{\Phi}\circ\Psi^{-1})\frac{\bar{\Psi}^{-1}_z}{\Psi^{-1}_z}}\big(\Psi(z_0)\big) 
    = \dfrac{\Psi_{z}(z_0)(\mu_{\Phi}(z_0)-\mu_{\Psi}(z_0))}{\bar{\Psi}_{z}(z_0)\big(1-\bar{\mu}_{\Psi}(z_0)\mu_{\Phi}(z_0)\big)}.
\end{equation}
Consequently, we have
\begin{equation}
    \abs{\mu_{\Phi\circ\Psi^{-1}}(\Psi(z_0))} \leq C\abs{\mu_{\Phi}(z_0)-\mu_{\Psi}(z_0)},
\end{equation}
where $C$ is a constant depending on $\Phi$ and $\Psi$. This shows that when $\mu_{\Phi}$ and $\mu_{\Psi}$ are close enough, the composition map $\Phi\circ\Psi^{-1}$ is close to conformal, and hence the errors $e_T$ and $e$ we adopt are good measurements of the error in the Beltrami coefficients.

\subsection{Example 1: A multiply-connected face mesh with 2 inner holes}
We first test our proposed PGQCM algorithm on a multiply-connected face mesh with 2 inner holes as shown in Fig.~\ref{fig:Ex1}. Fig.~\ref{fig:Ex1}(a) show the face mesh partitioned into four submeshes. We compute the free-boundary quasi-conformal parameterization for each mesh and then perform the partial welding and Koebe's iteration. Specifically, by welding the boundaries of the red submesh and the blue submesh together based on their partial correspondence and then transforming the inner boundary of the welded mesh to a circle, we obtain the updated boundary conditions as shown in Fig.~\ref{fig:Ex1}(b). Similarly, we obtain the updated boundary conditions of the green submesh and the magenta submesh as shown in Fig.~\ref{fig:Ex1}(c) by welding them together and transforming the inner boundary to a circle. After that, we weld the two welded shapes in Fig.~\ref{fig:Ex1}(b)--(c) according to their boundary correspondence and obtain the updated boundary conditions in Fig.~\ref{fig:Ex1}(d). It can be observed that the two inner holes remain close to circles, which demonstrates the efficacy of our parallel Koebe's iteration. Now, since the outer boundary of the updated shape is not circular, we perform the geodesic algorithm to transform it to a circle with other points lying inside it as shown in Fig.~\ref{fig:Ex1}(e). Finally, with the updated boundary conditions, we can compute the quasi-conformal parameterization for each submesh and obtain the global parameterization as shown in Fig.~\ref{fig:Ex1}(f). Table~\ref{tab:Ex1} records the mean absolute error between the prescribed Beltrami coefficient and the Beltrami coefficient of the parameterization result for each submesh, from which we see that the error is very small for all submeshes.

\subsection{Example 2: A multiply-connected face mesh with 3 inner holes}
We then test our algorithm on another multiply-connected face mesh with 3 inner holes as shown in Fig.~\ref{fig:Ex2}. Fig.~\ref{fig:Ex2}(a) shows the face mesh partitioned into 6 submeshes. We first compute the free-boundary quasi-conformal parameterization of each submesh respectively. Then, as shown in Fig.~\ref{fig:Ex2}(b)--(d), we weld three pairs of submesh boundaries and transform the inner boundaries of the results into circles. After that, we continue to perform welding to obtain the global mapping of the boundaries as shown in Fig.~\ref{fig:Ex2}(e)--(f). It can be observed that the inner boundaries remain very close to circles under the map. We then apply the geodesic algorithm to transform the outer boundary to a circle as shown in Fig.~\ref{fig:Ex2}(g). Finally, we solve the Laplace equation for each submesh with the updated boundary condition to obtain the global quasi-conformal parameterization as shown in Fig.~\ref{fig:Ex2}(h). The mean absolute error in the Beltrami coefficients is recorded in Table~\ref{tab:Ex2}, from which we can again see that the parameterization is very accurate.

\begin{figure}[t!]
\centering
  \includegraphics[width=\linewidth]{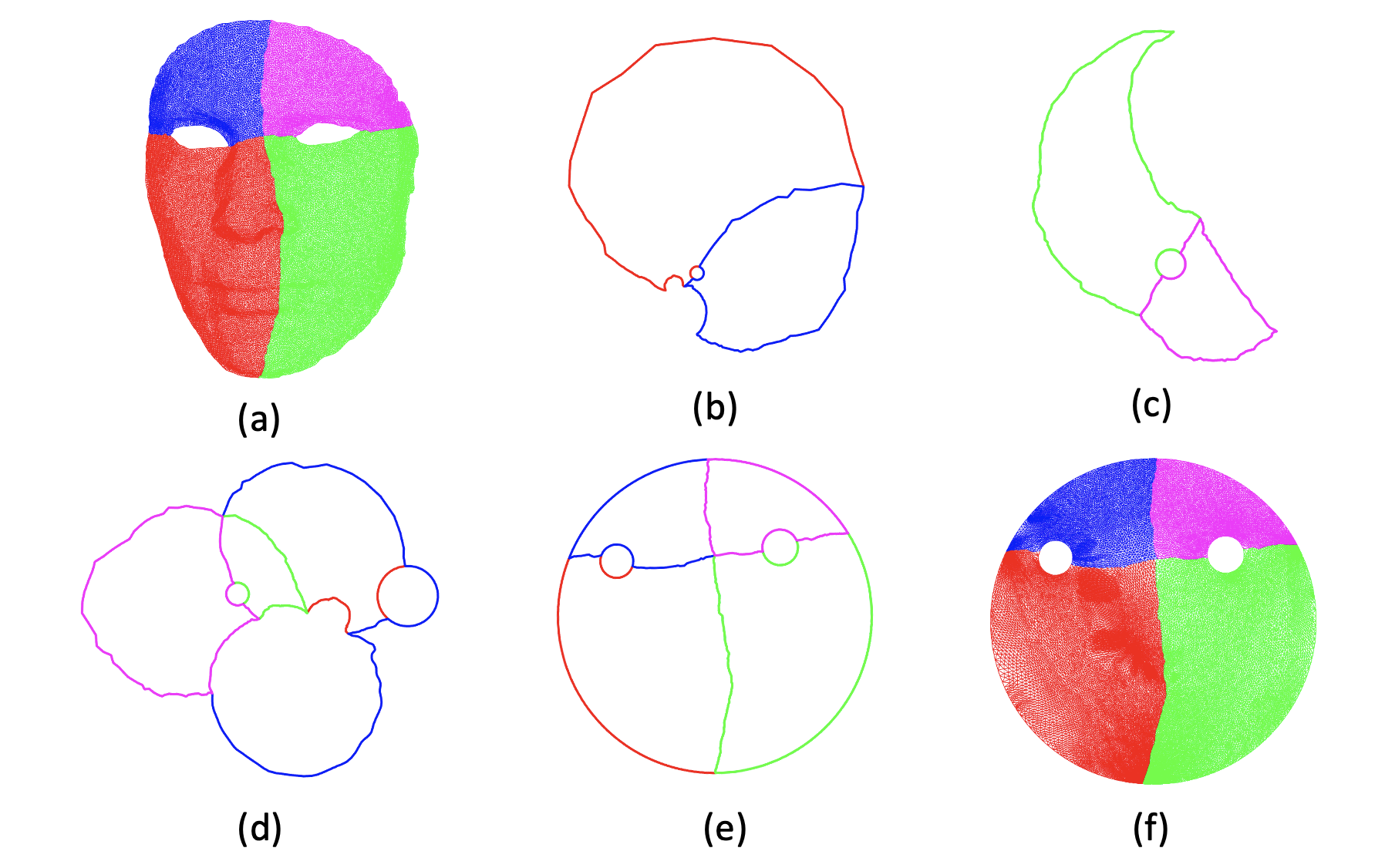}
\caption{Parameterizing a multiply-connected mesh with 2 holes using our PGQCM algorithm.}
\label{fig:Ex1}
\end{figure} 

\begin{table}[t!]
    \centering
    \begin{tabular}{c|c}
        Submesh & Error $e$\\ \hline 
        Submesh 1 & 0.0130\\
        Submesh 2 & 0.0106\\
        Submesh 3 & 0.0074\\
        Submesh 4 & 0.0088\\
    \end{tabular}
    \caption{Mean absolute error in Beltrami coefficients $\mu$ for each submesh in Fig.~\ref{fig:Ex1}.}
    \label{tab:Ex1}
\end{table}

\subsection{Example 3: A synthetic mesh with 4 inner holes}
We now consider parameterizing a synthetic multiply-connected mesh with 4 inner holes as shown in Fig.~\ref{fig:Ex3}. Fig.~\ref{fig:Ex3}(a) shows the original mesh partitioned into 8 submeshes. After computing the free-boundary quasi-conformal parameterization for each submesh, we weld 4 pairs of submesh boundaries and transform the inner boundary of each of them into a circle as shown in Fig.~\ref{fig:Ex3}(b)--(e). Then, we weld the results of Fig.~\ref{fig:Ex3}(b)--(c) into Fig.~\ref{fig:Ex3}(f) and those of Fig.~\ref{fig:Ex3}(d)--(e) into Fig.~\ref{fig:Ex3}(g). It can be observed that the inner boundaries are still very close to circles after the welding step. In Fig.~\ref{fig:Ex3}(h), we show the global boundary condition obtained by welding the results of Fig.~\ref{fig:Ex3}(f)--(g). We then transform the outer boundary into a circle to obtain the result in Fig.~\ref{fig:Ex3}(i). Finally, we solve that Laplace equation for each submesh with the updated boundary condition to obtain the global parameterization in Fig.~\ref{fig:Ex3}(j). As shown in Table~\ref{tab:Ex3}, the mean absolute error in the Beltrami coefficients is very small for all submeshes.

\begin{figure}[t!]
\centering
  \includegraphics[width=\linewidth]{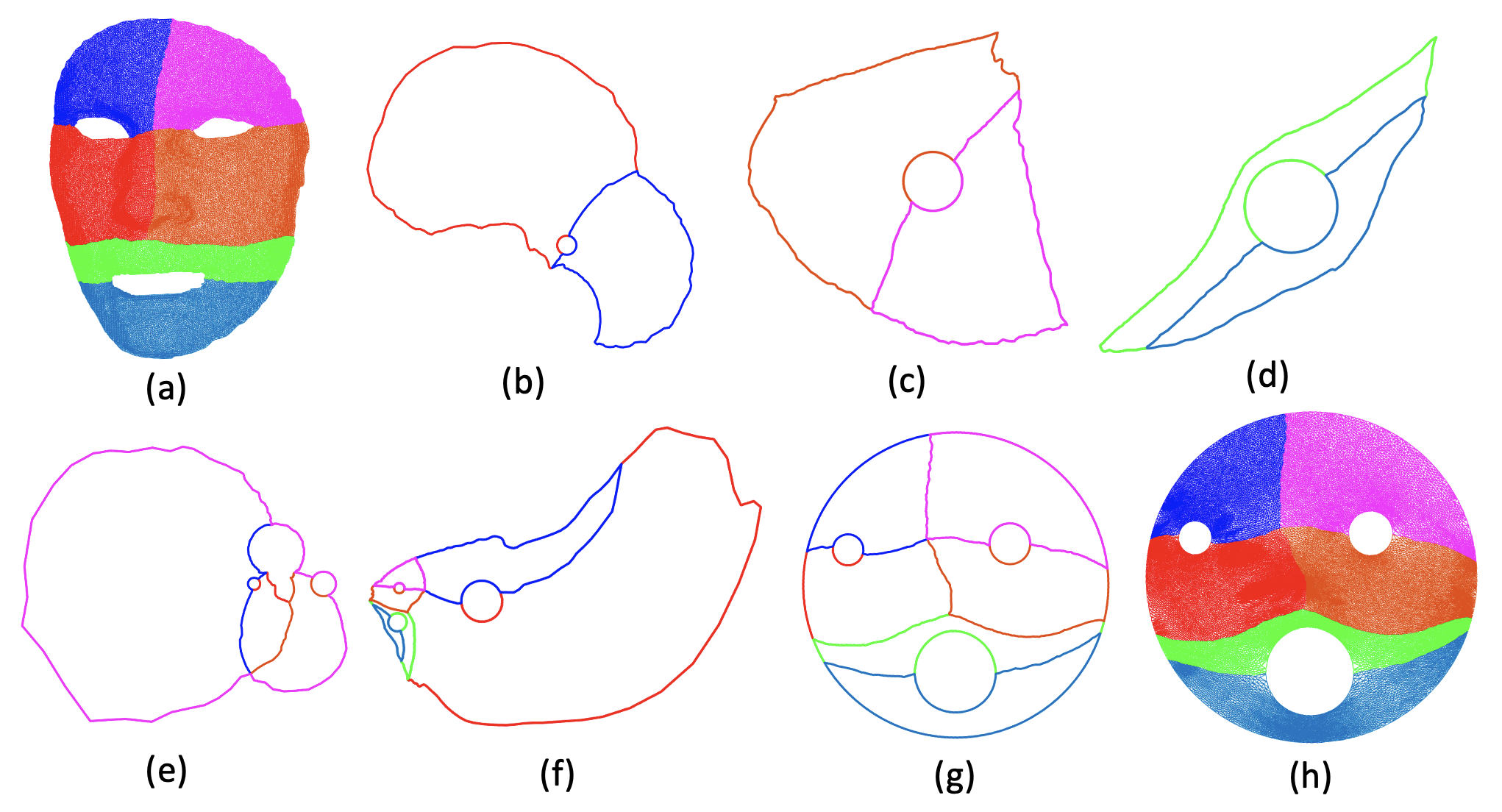}
\caption{Parameterizing a multiply-connected mesh with 3 holes using our PGQCM algorithm.}
\label{fig:Ex2}
\end{figure} 

\begin{table}[ht!]
    \centering
    \begin{tabular}{c|c}
        Submesh & Error $e$\\ \hline 
        Submesh 1 & 0.0213\\
        Submesh 2 & 0.0107\\
        Submesh 3 & 0.0139\\
        Submesh 4 & 0.0118\\
        Submesh 5 & 0.0131\\
        Submesh 6 & 0.0087\\
    \end{tabular}
    \caption{Mean absolute error in Beltrami coefficients $\mu$ for each submesh in Fig.~\ref{fig:Ex2}.}
    \label{tab:Ex2}
\end{table}

\begin{figure}[t!]
\centering
  \includegraphics[width=\linewidth]{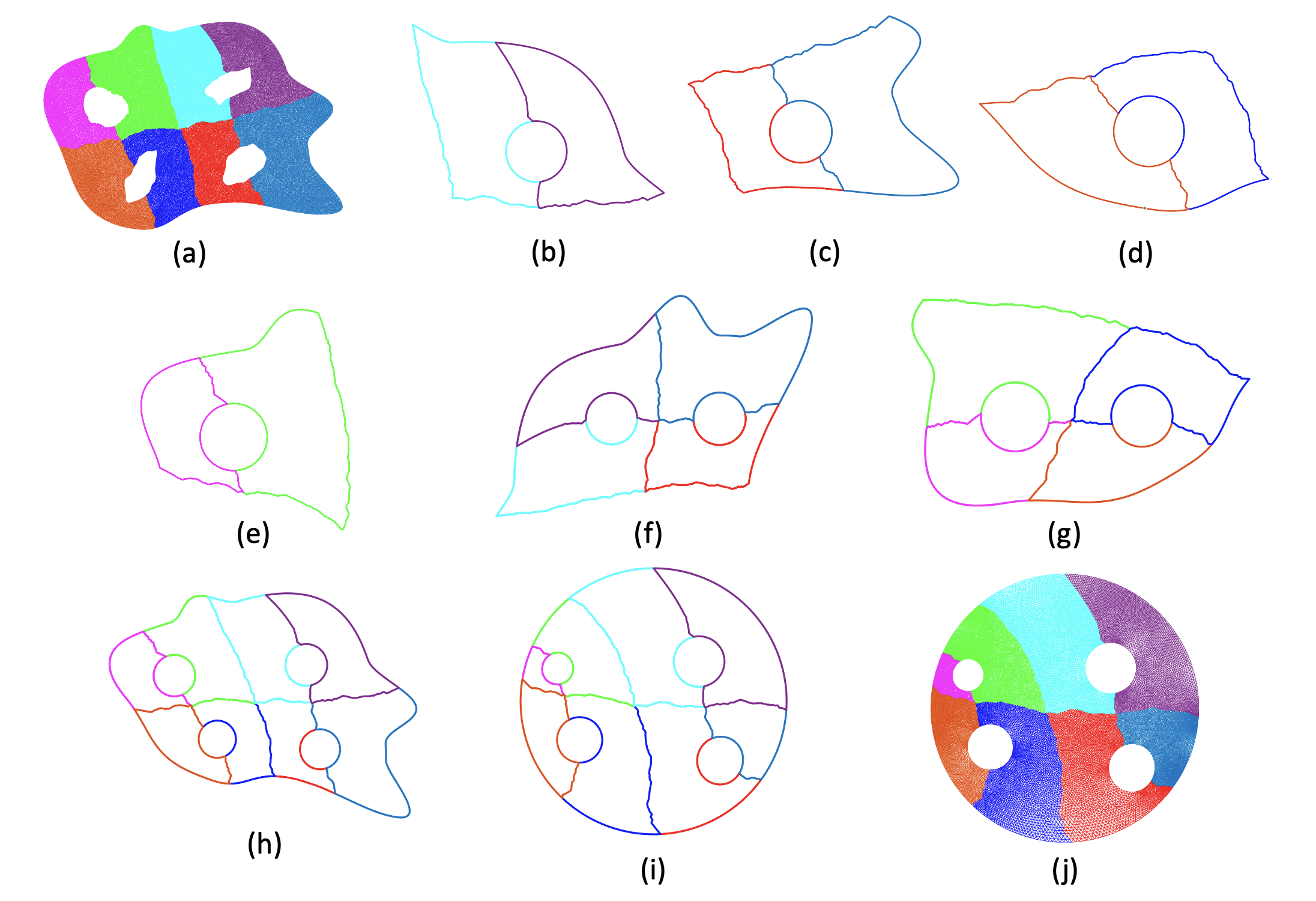}
\caption{Parameterizing a multiply-connected mesh with 4 holes using our PGQCM algorithm.}
\label{fig:Ex3}
\end{figure} 

\begin{table}[ht!]
    \centering
    \begin{tabular}{c|c}
        Submesh & Error $e$\\ \hline 
        Submesh 1 & 0.0064\\
        Submesh 2 & 0.0071\\
        Submesh 3 & 0.0089\\
        Submesh 4 & 0.0135\\
        Submesh 5 & 0.0054\\
        Submesh 6 & 0.0117\\
        Submesh 7 & 0.0056\\
        Submesh 8 & 0.0038\\
    \end{tabular}
    \caption{Mean absolute error in Beltrami coefficients $\mu$ for each submesh in Fig.~\ref{fig:Ex3}.}
    \label{tab:Ex3}
\end{table}

\subsection{Comparison between our proposed method and other parameterization methods}
After demonstrating the effectiveness of our proposed PGQCM method using various examples, we compare our method with other existing conformal and quasi-conformal parameterization methods in terms of the accuracy and efficiency. 

We first compare our proposed method with the QCMC iterative method for quasi-conformal parameterization~\cite{ho2016qcmc}. As shown in Table~\ref{tab:time_qcmc}, our method is significantly faster than the QCMC method by over 95\% on average for coarse and moderately dense meshes. For dense meshes, either our method is nearly 100 times faster or the QCMC method even fails to compute the desired mapping. This can be explained by the use of the divide-and-conquer strategy with parallelization in our algorithm. Also, the mean absolute error in the Beltrami coefficients of our method is generally much smaller than that of QCMC, especially for moderate and dense meshes. The experiments show that our method is more advantageous for computing quasi-conformal parameterization of multiply-connected surfaces.

\begin{table}[t]
    \centering
    \begin{tabular}{c|c|c|c|c|c}
        \multirow{2}{*}{Mesh} & \multirow{2}{*}{\# vertices} & \multicolumn{2}{c|}{PGQCM} & \multicolumn{2}{c}{QCMC} \\ \cline{3-6}  &  & Time (s) & Error $e$ & Time (s) & Error $e$ \\ \hline
        Amoeba 1 & 7322 & 0.1980 & 0.0389 & 7.8300 & 0.0420 \\
        Amoeba 2 & 27755 & 1.0229 & 0.0082 & 35.5800 & 0.0281 \\
        Alex & 13969 & 0.6515 & 0.0129 & 14.5116 & 0.0255 \\
        David 1 & 47550 & 0.9883 & 0.0108 & 28.4376 & 0.0234 \\
        David 2 & 48853 & 0.8251 & 0.0083 & 28.8781 & 0.0225 \\
        Face & 518890 & 14.5266 & 0.0014 & 1233.2295 & 0.0139 \\
        Stripe & 720150 & 21.6211 & 0.0024 & Failed & N/A \\
        Catenary & 1113041 & 37.8381 & 0.0015 & Failed & N/A \\
    \end{tabular}
    \caption{Comparison between PGQCM and QCMC~\cite{ho2016qcmc} for quasi-conformal parameterization of multiply-connected open surfaces in terms of the computational time and the mean absolute error in the Beltrami coefficients.}
    \label{tab:time_qcmc}
\end{table}

\begin{table}[t]
    \centering
    \begin{tabular}{c|c|c|c|c|c}
        \multirow{2}{*}{Mesh} & \multirow{2}{*}{\# vertices} & \multicolumn{2}{c|}{PGQCM} & \multicolumn{2}{c}{PACM} \\ \cline{3-6}  &  & Time (s) & Error $e$ & Time (s) & Error $e$ \\ \hline
        Amoeba 1 & 7322 & 0.1899 & 0.0173 & 0.5232 & 0.0106 \\
        Amoeba 2 & 27755 & 1.0185 & 0.0078 & 5.1781 & 0.0044 \\
        Alex & 13969 & 0.6559 & 0.0127 & 2.5095 & 0.0218 \\
        David 1 & 47550 & 0.8107 & 0.0106 & 3.9559 & 0.0213 \\
        David 2 & 48853 & 0.7783 & 0.0079 & 3.3490 & 0.0086 \\
        Face & 518890 & 11.6979 & 0.0013 & 100.2848 & 0.0041 \\
        Stripe & 720150 & 17.7237 & 0.0029 & 214.6671 & 0.0198 \\
        Catenary & 1113041 & 31.8079 & 0.0014 & 349.2447 & 0.0186 \\
    \end{tabular}
    \caption{Comparison between PGQCM and PACM~\cite{choi2021efficient} for conformal parameterization of multiply-connected open surfaces in terms of the computational time and the mean absolute error in the Beltrami coefficients.}
    \label{tab:time_pacm}
\end{table}

Next, we compare our proposed method with the recently developed PACM algorithm~\cite{choi2021efficient}. In particular, since the PACM method only works for the conformal parameterizations of multiply-connected surfaces, here we set the target Beltrami coefficient in our algorithm to be $\mu \equiv 0$ and compute conformal parameterizations for the comparison. As shown in Table~\ref{tab:time_pacm}, our method is faster than the PACM method by over 80\% on average. Also, the mean absolute error in the Beltrami coefficients of our method is generally much smaller than that of PACM for moderate and dense meshes. This shows that our method is not only useful for quasi-conformal parameterization but also for conformal parameterization of multiply-connected surfaces.

To further explain the significant improvement in the computational efficiency achieved by our method, note that computing a free-boundary quasi-conformal map for a global mesh requires solving a large sparse linear system. More specifically, for the global mapping of a triangle mesh with $N$ vertices, one needs to solve a linear system of size $2N \times 2N$. However, if we partition the mesh into $k$ submeshes of the same size, we only need to solve $k$ much smaller linear systems of size $\frac{2N}{k}\times \frac{2N}{k}$. Suppose the original computation cost is $C$. By partitioning the mesh into submeshes, we reduce the cost to $\frac{C}{k^{1/2}}$. Since partitioning the surface enables us to apply parallel computing to compute the parameterization, the computational time can be further reduced by a large extent. Also, when the mesh size is extremely large, other existing global mapping methods may fail due to the extremely large linear systems involved. By contrast, by partitioning the surface, in our method we only need to handle smaller linear systems, which are much easier to solve. Note that we also need to take the additional computation cost of welding into account when analyzing the total computational cost. However, since the welding step only involves the boundary points and the complexity of the geodesic algorithm is $O(mn)$, where $m$ is the number of points that determine the map and $n$ is the number of points we want to update, the computational cost of the welding step is considerably less than that of the quasi-conformal parameterization step. Besides, as discussed in detail in Section~\ref{sec:4.5}, partitioning the mesh allows us to perform the parallel Koebe's iteration, which also help reduce the total computational cost. Altogether, our method greatly accelerates the computation of conformal and quasi-conformal parameterizations for multiply-connected surfaces. 

\section{Applications}
\label{sec:6}
\subsection{Texture mapping}
Our multiply-connected quasi-conformal parameterization method can be used for texture mapping on multiply-connected open surfaces. More specifically, after mapping a given multiply-connected 3D surface to a 2D circular domain, we can design the texture on the 2D circular domain freely, and then map the texture onto the mesh using the inverse mapping of the parameterization. 

Since conformal parameterizations preserve local geometry, they are commonly used for texture mapping so that the local distortion of the designed texture is small. Similar to~\cite{choi2021efficient}, we can set the prescribed Beltrami coefficient as $0$ in our proposed algorithm and compute a conformal parameterization for texture mapping. An example is given in the top row of Fig.~\ref{fig:texture}. Here, we first parameterize a multiply-connected human face mesh onto the 2D circular domain conformally using our method. Then, we design a checkerboard texture on the 2D circular domain and map the texture back onto the mesh via the parameterization. It can be observed that the right angles in the checkerboard pattern are well-preserved on the human face, which indicates that the local geometry is not distorted. 

Moreover, since our method is capable of computing quasi-conformal parameterizations, it grants us more flexibility in the texture mapping design. Specifically, we can prescribe the level of local geometric distortion at any point freely using the input Beltrami coefficient, which allows us to design textures with different visual effects via the quasi-conformal parameterization. An example is given in the bottom row of Fig.~\ref{fig:texture}. Note that the orthogonality of the checkerboard texture is well-preserved at the nose of the human face, while an angular distortion in the checkerboard pattern can be clearly observed at the chin and the forehead. This demonstrates the possibility of achieving different texture mapping effects using our proposed method.

\begin{figure}[t]
\centering
  \includegraphics[width=1.0\linewidth]{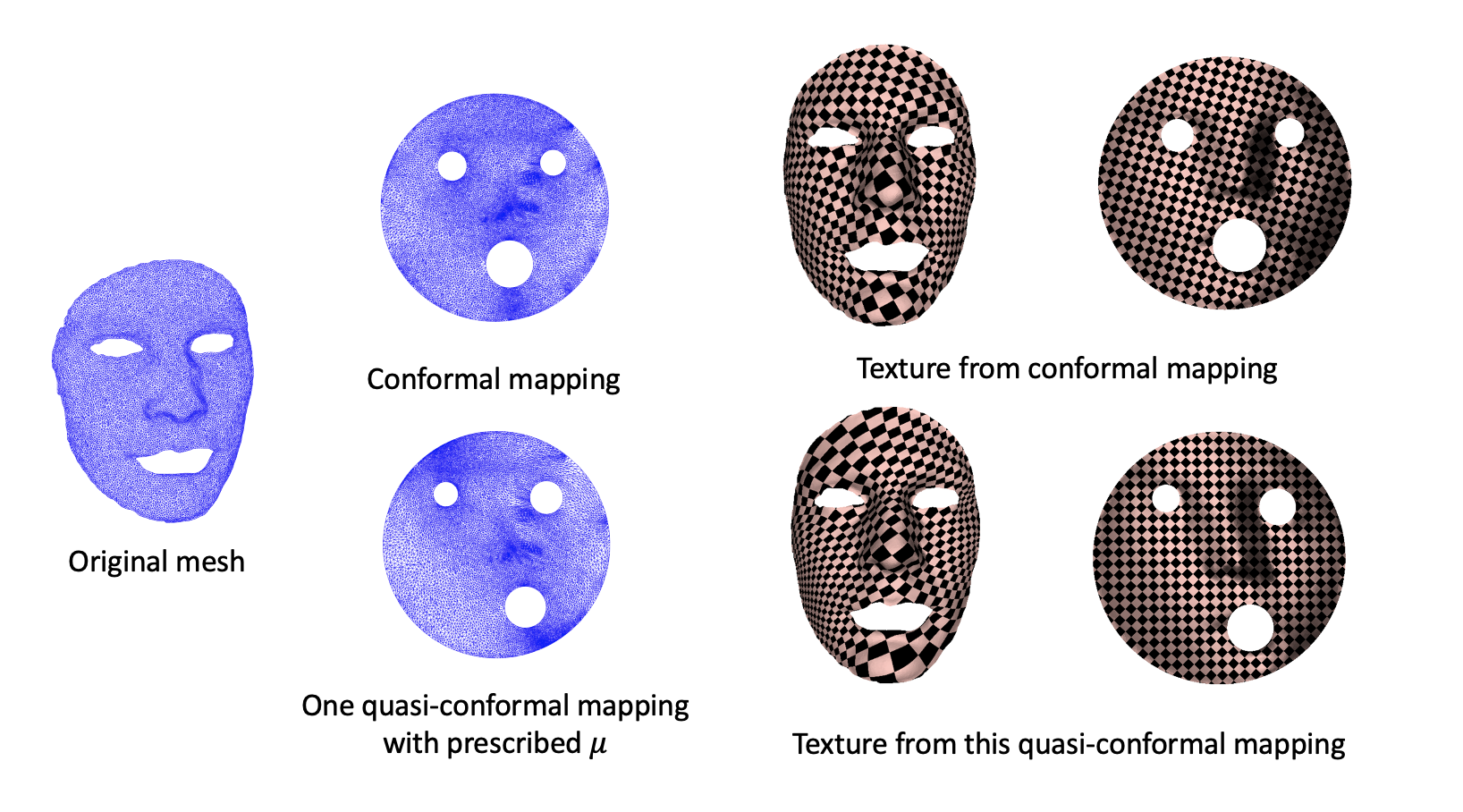}
\caption{Texture mapping using the conformal and quasi-conformal parameterizations obtained by our proposed algorithm.}
\label{fig:texture} 
\end{figure} 

\subsection{Remeshing}
Similar to~\cite{choi2021efficient}, our algorithm can be used to perform surface remeshing. Given a 3D multiply-connected open surface, we first parameterize it onto a standard 2D circular domain using our proposed method. Then, we can design a new mesh structure in the circular domain, and finally obtain the remeshed 3D surface with the new mesh structure using the inverse mapping. Two examples are given in Fig.~\ref{fig:remeshing}. In both examples, the remeshing in the circular domain is done using the \texttt{ddiff} and \texttt{dcircle} functions in the DistMesh toolbox~\cite{persson2004simple}. We remark that analogous to the texture mapping application described above, here we can achieve different remeshing effects by using different choices of the Beltrami coefficient in computing the quasi-conformal parameterization.

\begin{figure}[t]
\centering
  \includegraphics[width=\linewidth]{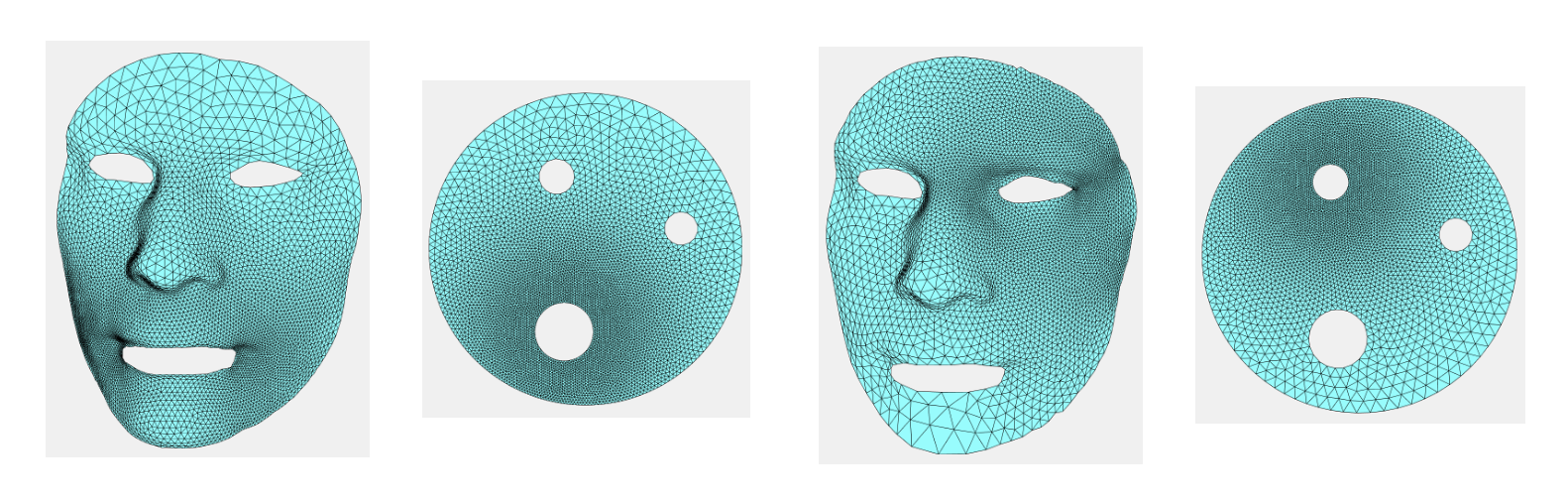}
\caption{Remeshing a 3D multiply-connected open surface using our proposed algorithm.}
\label{fig:remeshing}
\end{figure} 

\section{Discussion}
\label{sec:7}
In this paper, we have developed a novel parallelizable method for computing the global quasi-conformal parameterization of multiply-connected surfaces. Given any multiply-connected open surface and any prescribed Beltrami coefficient $\mu$, our method computes a quasi-conformal parameterization onto a 2D circular domain in a parallelizable manner. In particular, with the prescribed Beltrami coefficient being $\mu = 0$, conformal parameterizations can be efficiently obtained. When compared to other existing conformal and quasi-conformal parameterization methods for multiply-connected surfaces, our proposed method is more advantageous in both the efficiency and accuracy.

Below, we discuss two possible future research directions on further improving the performance of our method.

\subsection{Area distortion}
One known issue of conformal mapping is that the area distortion may be significant~\cite{kharevych2006discrete}, and this also happens in the quasi-conformal case~\cite{qiu2020inconsistent}. Therefore, it is natural to ask how we can reduce the area distortion of the quasi-conformal parameterization without altering the Beltrami coefficient $\mu$. Given a multiply-connected mesh $\mathcal{S} = (\mathcal{V},\mathcal{F})$ and the global quasi-conformal parameterization $\Phi:\mathcal{S}\to\mathbb{R}^2$ obtained by our PGQCM algorithm, we define the area distortion of $\Phi$ on a triangle face $T\in\mathcal{F}$ as follows~\cite{choi2020area,choi2020parallelizable}:
\begin{equation}
    d_{\text{area}}(T) = \log \dfrac{\text{Area}(\Phi(T))/(\Sigma_{T'\in\mathcal{F}}\text{Area}(\Phi(T')))}{\text{Area}(T)/(\Sigma_{T'\in\mathcal{F}}\text{Area}(T'))}.
\end{equation}
Specifically, the numerator of $d_{\text{Area}}(T)$ measures the ratio of the area of $T$ to the surface area of $\mathcal{S}$, the denominator measures the ratio of the area of $\Phi(T)$ to the total area of $\Phi(\mathcal{S})$, and $d_{\text{Area}}(T)$ is the logged area ratio. Note that $d_{\text{Area}}(T) \approx 0$ indicates that the area distortion is small, while a large $d_{\text{Area}}(T)$ indicates that the area distortion is large. If the surface area of $\mathcal{S}$ and $\Phi(\mathcal{S})$ is not equal, we can simply compose a normalization map $cz$ for some constant $c$ to make them equal.\\
\indent Similar to~\cite{choi2020parallelizable,choi2021efficient}, we can consider reducing the area distortion of our parameterization by composing an automorphism of the unit disk after obtaining the global parameterization in the last step in Algorithm~\ref{alg:4}. Note that the Beltrami coefficient will not be changed by an automorphism of the unit disk as suggested by the composition formula in Equation~\eqref{eq:3}. Also, note that M\"{o}bius transformations always map circles to circles or straight lines, and in our case we compose an automorphism of the unit circle and so we will only have the former case. Therefore, the circular inner boundaries will be mapped to circles under the automorphism. We can search for an optimal automorphism
    $f(z) = \dfrac{z-\alpha}{1-\bar{\alpha}z}$, 
where $\alpha \in \mathbb{C}$ satisfies $\abs{\alpha}<1$, such that $f\circ\Phi$ minimizes the area distortion $\sum_{T\in \mathcal{F}}d_{\text{Area}}(T)$. The map $f\circ \Phi$ will then be the desired global quasi-conformal parameterization with area distortion reduced. However, note that this approach involves handling the global mesh and hence may not be computationally efficient or feasible if the input mesh is dense. In our future work, we plan to develop methods for reducing the area distortion in a parallelizable manner. We also plan to explore other possible measures of the area distortion and optimization methods for improving the performance of the area correction. 

\subsection{Acceleration of our algorithm}
In recent decades, parallel computing has been widely studied and applied for high-performance computing on large datasets. While we have demonstrated the efficiency of our parallelizable algorithm for computing quasi-conformal parameterizations of meshes, especially for large meshes, there is still room to further speed up the computation as outlined below. 

In our experiments, for simplicity we perform the computation using the Parallel Computing Toolbox in MATLAB. Although we can already achieve a significant improvement in the performance when compared to the prior methods, the Parallel Computing Toolbox in MATLAB may not be the best choice for our proposed method. For instance, as mentioned by~\cite{choi2020parallelizable}, some of the MATLAB built-in functions such as \texttt{fminunc} are not parallelizable under the parallel computing framework of MATLAB, and so MATLAB may not allow us to fully exploit parallelization in some steps of our proposed method. Therefore, we plan to consider other scientific computing software and platforms more specialized in parallel computing in our future work and evaluate the performance of our proposed method.

In our algorithm, we adopt a finite element approach to compute the quasi-conformal parameterizations, which requires us to solve large sparse linear systems accurately and efficiently. In our implementation, the backslash operator ($\backslash$) in MATLAB is used to solve linear systems. Besides this convenient built-in function in MATLAB, there are some alternatives that we may consider for solving the linear systems. One notable example is the combinatorial multigrid (CMG) method in~\cite{koutis2011combinatorial}, which is a hybrid graph-theoretic algebraic multigrid solver. Also, parallel sparse linear system solvers, such as the solvers by Peng and Spielman~\cite{peng2014efficient} and by Koutis and Miller~\cite{koutis2007linear}, may be considered and incorporated to our algorithm for further improving the computational efficiency.

\bibliographystyle{siam}
\bibliography{citation.bib}

\begin{thebibliography}{10}

\bibitem{ahlfors1960riemann}
{\sc L.~Ahlfors and L.~Bers}, {\em Riemann's mapping theorem for variable
  metrics}, Annals of Mathematics,  (1960), pp.~385--404.

\bibitem{angenent1999laplace}
{\sc S.~Angenent, S.~Haker, A.~Tannenbaum, and R.~Kikinis}, {\em On the
  {L}aplace-{B}eltrami operator and brain surface flattening}, IEEE
  Transactions on Medical Imaging, 18 (1999), pp.~700--711.

\bibitem{astala2008elliptic}
{\sc K.~Astala, T.~Iwaniec, and G.~Martin}, {\em Elliptic Partial Differential
  Equations and Quasiconformal Mappings in the Plane (PMS-48)}, Princeton
  University Press, 2008.

\bibitem{bobenko2016discrete}
{\sc A.~I. Bobenko, S.~Sechelmann, and B.~Springborn}, {\em Discrete conformal
  maps: Boundary value problems, circle domains, {F}uchsian and {S}chottky
  uniformization}, in Advances in Discrete Differential Geometry, Springer,
  Berlin, Heidelberg, 2016, pp.~1--56.

\bibitem{chien2016bounded}
{\sc E.~Chien, Z.~Levi, and O.~Weber}, {\em Bounded distortion parametrization
  in the space of metrics}, ACM Transactions on Graphics, 35 (2016), pp.~1--16.

\bibitem{choi2021efficient}
{\sc G.~P.~T. Choi}, {\em Efficient conformal parameterization of
  multiply-connected surfaces using quasi-conformal theory}, Journal of
  Scientific Computing, 87 (2021), pp.~1--19.

\bibitem{choi2020tooth}
{\sc G.~P.~T. Choi, H.~L. Chan, R.~Yong, S.~Ranjitkar, A.~Brook, G.~Townsend,
  K.~Chen, and L.~M. Lui}, {\em Tooth morphometry using quasi-conformal
  theory}, Pattern Recognition, 99 (2020), p.~107064.

\bibitem{choi2020area}
{\sc G.~P.~T. Choi, B.~Chiu, and C.~H. Rycroft}, {\em Area-preserving mapping
  of {3D} carotid ultrasound images using density-equalizing reference map},
  IEEE Transactions on Biomedical Engineering, 67 (2020), pp.~1507--1517.

\bibitem{choi2016spherical}
{\sc G.~P.-T. Choi, K.~T. Ho, and L.~M. Lui}, {\em Spherical conformal
  parameterization of genus-0 point clouds for meshing}, SIAM Journal on
  Imaging Sciences, 9 (2016), pp.~1582--1618.

\bibitem{choi2020parallelizable}
{\sc G.~P.~T. Choi, Y.~Leung-Liu, X.~Gu, and L.~M. Lui}, {\em Parallelizable
  global conformal parameterization of simply-connected surfaces via partial
  welding}, SIAM Journal on Imaging Sciences, 13 (2020), pp.~1049--1083.

\bibitem{choi2022free}
{\sc G.~P.~T. Choi, Y.~Liu, and L.~M. Lui}, {\em Free-boundary conformal
  parameterization of point clouds}, Journal of Scientific Computing, 90
  (2022), pp.~1--26.

\bibitem{choi2018linear}
{\sc G.~P.-T. Choi and L.~M. Lui}, {\em A linear formulation for disk conformal
  parameterization of simply-connected open surfaces}, Advances in
  Computational Mathematics, 44 (2018), pp.~87--114.

\bibitem{choi2016fast}
{\sc G.~P.-T. Choi, M.~H.-Y. Man, and L.~M. Lui}, {\em Fast spherical
  quasiconformal parameterization of genus-$0 $ closed surfaces with
  application to adaptive remeshing}, Geometry, Imaging and Computing, 3
  (2016), pp.~1--29.

\bibitem{choi2020shape}
{\sc G.~P.~T. Choi, D.~Qiu, and L.~M. Lui}, {\em Shape analysis via
  inconsistent surface registration}, Proceedings of the Royal Society A, 476
  (2020), p.~20200147.

\bibitem{choi2018density}
{\sc G.~P.~T. Choi and C.~H. Rycroft}, {\em Density-equalizing maps for simply
  connected open surfaces}, SIAM Journal on Imaging Sciences, 11 (2018),
  pp.~1134--1178.

\bibitem{choi2015flash}
{\sc P.~T. Choi, K.~C. Lam, and L.~M. Lui}, {\em {FLASH}: Fast landmark aligned
  spherical harmonic parameterization for genus-0 closed brain surfaces}, SIAM
  Journal on Imaging Sciences, 8 (2015), pp.~67--94.

\bibitem{choi2015fast}
{\sc P.~T. Choi and L.~M. Lui}, {\em Fast disk conformal parameterization of
  simply-connected open surfaces}, Journal of Scientific Computing, 65 (2015),
  pp.~1065--1090.

\bibitem{crowdy2005schwarz}
{\sc D.~Crowdy}, {\em The {S}chwarz--{C}hristoffel mapping to bounded multiply
  connected polygonal domains}, Proceedings of the Royal Society A, 461 (2005),
  pp.~2653--2678.

\bibitem{crowdy2007schwarz}
{\sc D.~Crowdy}, {\em Schwarz--christoffel mappings to unbounded multiply
  connected polygonal regions}, in Mathematical Proceedings of the Cambridge
  Philosophical Society, vol.~142, Cambridge University Press, 2007,
  pp.~319--339.

\bibitem{crowdy2006conformal}
{\sc D.~Crowdy and J.~Marshall}, {\em Conformal mappings between canonical
  multiply connected domains}, Computational Methods and Function Theory, 6
  (2006), pp.~59--76.

\bibitem{cui2019spherical}
{\sc L.~Cui, X.~Qi, C.~Wen, N.~Lei, X.~Li, M.~Zhang, and X.~Gu}, {\em Spherical
  optimal transportation}, Computer-Aided Design, 115 (2019), pp.~181--193.

\bibitem{desbrun2002intrinsic}
{\sc M.~Desbrun, M.~Meyer, and P.~Alliez}, {\em Intrinsic parameterizations of
  surface meshes}, Computer Graphics Forum, 21 (2002), pp.~209--218.

\bibitem{floater2005surface}
{\sc M.~S. Floater and K.~Hormann}, {\em Surface parameterization: a tutorial
  and survey}, Advances in multiresolution for geometric modelling,  (2005),
  pp.~157--186.

\bibitem{garcia1997path}
{\sc J.~J. Garcia-Luna-Aceves and S.~Murthy}, {\em A path-finding algorithm for
  loop-free routing}, IEEE/ACM Transactions on Networking, 5 (1997),
  pp.~148--160.

\bibitem{gardiner2000quasiconformal}
{\sc F.~P. Gardiner and N.~Lakic}, {\em Quasiconformal {T}eichm\"uller theory},
  vol.~76, American Mathematical Society, 2000.

\bibitem{giri2020open}
{\sc A.~Giri, G.~P.~T. Choi, and L.~Kumar}, {\em Open and closed anatomical
  surface description via hemispherical area-preserving map}, Signal
  Processing, 180 (2021), p.~107867.

\bibitem{RiemannMapper}
{\sc D.~Gu}, {\em {RiemannMapper} : A mesh parameterization toolkit}.
\newblock \url{https://www3.cs.stonybrook.edu/~gu/software/RiemannMapper/}.

\bibitem{gu2004genus}
{\sc X.~Gu, Y.~Wang, T.~F. Chan, P.~M. Thompson, and S.-T. Yau}, {\em Genus
  zero surface conformal mapping and its application to brain surface mapping},
  IEEE Transactions on Medical Imaging, 23 (2004), pp.~949--958.

\bibitem{gu2003global}
{\sc X.~Gu and S.-T. Yau}, {\em Global conformal surface parameterization}, in
  Proceedings of the 2003 Eurographics/ACM SIGGRAPH Symposium on Geometry
  Processing, 2003, pp.~127--137.

\bibitem{haker2000conformal}
{\sc S.~Haker, S.~Angenent, A.~Tannenbaum, R.~Kikinis, G.~Sapiro, and
  M.~Halle}, {\em Conformal surface parameterization for texture mapping}, IEEE
  Transactions on Visualization and Computer Graphics, 6 (2000), pp.~181--189.

\bibitem{henrici1993applied}
{\sc P.~Henrici}, {\em Applied and computational complex analysis, Volume 3:
  Discrete Fourier analysis, Cauchy integrals, construction of conformal maps,
  univalent functions}, vol.~41, John Wiley \& Sons, 1993.

\bibitem{ho2016qcmc}
{\sc K.~T. Ho and L.~M. Lui}, {\em {QCMC}: quasi-conformal parameterizations
  for multiply-connected domains}, Advances in Computational Mathematics, 42
  (2016), pp.~279--312.

\bibitem{hormann2007mesh}
{\sc K.~Hormann, B.~L{\'e}vy, and A.~Sheffer}, {\em Mesh parameterization:
  Theory and practice}, ACM SIGGRAPH 2007 Course Notes,  (2007).

\bibitem{hutchinson1991computing}
{\sc J.~E. Hutchinson}, {\em Computing conformal maps and minimal surfaces}, in
  Theoretical and Numerical Aspects of Geometric Variational Problems,
  Australian National University, Mathematical Sciences Institute, 1991,
  pp.~140--161.

\bibitem{jin2008discrete}
{\sc M.~Jin, J.~Kim, F.~Luo, and X.~Gu}, {\em Discrete surface ricci flow},
  IEEE Transactions on Visualization and Computer Graphics, 14 (2008),
  pp.~1030--1043.

\bibitem{kharevych2006discrete}
{\sc L.~Kharevych, B.~Springborn, and P.~Schr{\"o}der}, {\em Discrete conformal
  mappings via circle patterns}, ACM Transactions on Graphics, 25 (2006),
  pp.~412--438.

\bibitem{koebe1910konforme}
{\sc P.~Koebe}, {\em {\"U}ber die konforme abbildung mehrfach
  zusammenh{\"a}ngender bereiche}, Jahresbericht der Deutschen
  Mathematiker-Vereinigung, 19 (1910), pp.~339--348.

\bibitem{koutis2007linear}
{\sc I.~Koutis and G.~L. Miller}, {\em A linear work, $o(n^{1/6})$ time,
  parallel algorithm for solving planar {L}aplacians}, in 8th Annual ACM-SIAM
  Symposium on Discrete Algorithms, SODA 2007, Association for Computing
  Machinery, 2007, pp.~1002--1011.

\bibitem{koutis2011combinatorial}
{\sc I.~Koutis, G.~L. Miller, and D.~Tolliver}, {\em Combinatorial
  preconditioners and multilevel solvers for problems in computer vision and
  image processing}, Computer Vision and Image Understanding, 115 (2011),
  pp.~1638--1646.

\bibitem{kropf2014conformal}
{\sc E.~Kropf, X.~Yin, S.-T. Yau, and X.~D. Gu}, {\em Conformal
  parameterization for multiply connected domains: Combining finite elements
  and complex analysis}, Engineering with Computers, 30 (2014), pp.~441--455.

\bibitem{kuhnau1983numerische}
{\sc R.~K{\"u}hnau}, {\em Numerische realisierung konformer abbildungen durch
  ``interpolation''}, ZAMM-Journal of Applied Mathematics and
  Mechanics/Zeitschrift f{\"u}r Angewandte Mathematik und Mechanik, 63 (1983),
  pp.~631--637.

\bibitem{lai2014folding}
{\sc R.~Lai, Z.~Wen, W.~Yin, X.~Gu, and L.~M. Lui}, {\em Folding-free global
  conformal mapping for genus-0 surfaces by harmonic energy minimization},
  Journal of Scientific Computing, 58 (2014), pp.~705--725.

\bibitem{lam2014landmark}
{\sc K.~C. Lam and L.~M. Lui}, {\em Landmark-and intensity-based registration
  with large deformations via quasi-conformal maps}, SIAM Journal on Imaging
  Sciences, 7 (2014), pp.~2364--2392.

\bibitem{lawrynowicz2006quasiconformal}
{\sc J.~Lawrynowicz}, {\em Quasiconformal Mappings in the Plane: Parametrical
  Methods}, vol.~978, Springer, 2006.

\bibitem{lehto1973quasiconformal}
{\sc O.~Lehto and K.~I. Virtanen}, {\em Quasiconformal mappings in the plane},
  vol.~126, Citeseer, 1973.

\bibitem{levy2002least}
{\sc B.~L{\'e}vy, S.~Petitjean, N.~Ray, and J.~Maillot}, {\em Least squares
  conformal maps for automatic texture atlas generation}, ACM Transactions on
  Graphics, 21 (2002), pp.~362--371.

\bibitem{lipman2012bounded}
{\sc Y.~Lipman}, {\em Bounded distortion mapping spaces for triangular meshes},
  ACM Transactions on Graphics, 31 (2012), pp.~1--13.

\bibitem{lui2013texture}
{\sc L.~M. Lui, K.~C. Lam, T.~W. Wong, and X.~Gu}, {\em Texture map and video
  compression using {B}eltrami representation}, SIAM Journal on Imaging
  Sciences, 6 (2013), pp.~1880--1902.

\bibitem{lui2014teichmuller}
{\sc L.~M. Lui, K.~C. Lam, S.-T. Yau, and X.~Gu}, {\em Teichm\"uller mapping
  ({T}-map) and its applications to landmark matching registration}, SIAM
  Journal on Imaging Sciences, 7 (2014), pp.~391--426.

\bibitem{lui2012optimization}
{\sc L.~M. Lui, T.~W. Wong, W.~Zeng, X.~Gu, P.~M. Thompson, T.~F. Chan, and
  S.-T. Yau}, {\em Optimization of surface registrations using {B}eltrami
  holomorphic flow}, Journal of Scientific Computing, 50 (2012), pp.~557--585.

\bibitem{luo2004combinatorial}
{\sc F.~Luo}, {\em Combinatorial {Y}amabe flow on surfaces}, Communications in
  Contemporary Mathematics, 6 (2004), pp.~765--780.

\bibitem{marshall2012conformal}
{\sc D.~E. Marshall}, {\em Conformal welding for finitely connected regions},
  Computational Methods and Function Theory, 11 (2012), pp.~655--669.

\bibitem{Marshall1987compositions}
{\sc D.~E. Marshall and J.~A. Morrow}, {\em Compositions of slit mappings},
  manuscript,  (1987).

\bibitem{marshall2007convergence}
{\sc D.~E. Marshall and S.~Rohde}, {\em Convergence of a variant of the zipper
  algorithm for conformal mapping}, SIAM Journal on Numerical Analysis, 45
  (2007), pp.~2577--2609.

\bibitem{meng2016tempo}
{\sc T.~W. Meng, G.~P.-T. Choi, and L.~M. Lui}, {\em {TEMPO}: Feature-endowed
  {T}eich\"muller extremal mappings of point clouds}, SIAM Journal on Imaging
  Sciences, 9 (2016), pp.~1922--1962.

\bibitem{nasser2020plgcirmap}
{\sc M.~M.~S. Nasser}, {\em {PlgCirMap}: A {MATLAB} toolbox for computing
  conformal mappings from polygonal multiply connected domains onto circular
  domains}, SoftwareX, 11 (2020), p.~100464.

\bibitem{ng2014teichmuller}
{\sc T.~C. Ng, X.~Gu, and L.~M. Lui}, {\em Teichm{\"u}ller extremal map of
  multiply-connected domains using {B}eltrami holomorphic flow}, Journal of
  Scientific Computing, 60 (2014), pp.~249--275.

\bibitem{peng2014efficient}
{\sc R.~Peng and D.~A. Spielman}, {\em An efficient parallel solver for {SDD}
  linear systems}, in Proceedings of the forty-sixth annual ACM symposium on
  Theory of computing, 2014, pp.~333--342.

\bibitem{persson2004simple}
{\sc P.-O. Persson and G.~Strang}, {\em A simple mesh generator in matlab},
  SIAM Review, 46 (2004), pp.~329--345.

\bibitem{pfluger1960ueber}
{\sc V.~A. Pfluger}, {\em Ueber die konstruktion riemannscher flachen durch
  verheftung}, The Journal of the Indian Mathematical Society, 24 (1960),
  pp.~401--412.

\bibitem{pinkall1993computing}
{\sc U.~Pinkall and K.~Polthier}, {\em Computing discrete minimal surfaces and
  their conjugates}, Experimental Mathematics, 2 (1993), pp.~15--36.

\bibitem{qiu2019computing}
{\sc D.~Qiu, K.-C. Lam, and L.-M. Lui}, {\em Computing quasi-conformal folds},
  SIAM Journal on Imaging Sciences, 12 (2019), pp.~1392--1424.

\bibitem{qiu2020inconsistent}
{\sc D.~Qiu and L.~M. Lui}, {\em Inconsistent surface registration via
  optimization of mapping distortions}, Journal of Scientific Computing, 83
  (2020), pp.~1--31.

\bibitem{sawhney2017boundary}
{\sc R.~Sawhney and K.~Crane}, {\em Boundary first flattening}, ACM
  Transactions on Graphics, 37 (2017), pp.~1--14.

\bibitem{shaqfa2021spherical}
{\sc M.~Shaqfa, G.~P.~T. Choi, and K.~Beyer}, {\em Spherical cap harmonic
  analysis ({SCHA}) for characterising the morphology of rough surface
  patches}, Powder Technology, 393 (2021), pp.~837--856.

\bibitem{sharon20062d}
{\sc E.~Sharon and D.~Mumford}, {\em 2d-shape analysis using conformal
  mapping}, International Journal of Computer Vision, 70 (2006), pp.~55--75.

\bibitem{sheffer2001parameterization}
{\sc A.~Sheffer and E.~de~Sturler}, {\em Parameterization of faceted surfaces
  for meshing using angle-based flattening}, Engineering with Computers, 17
  (2001), pp.~326--337.

\bibitem{sheffer2005abf++}
{\sc A.~Sheffer, B.~L{\'e}vy, M.~Mogilnitsky, and A.~Bogomyakov}, {\em {ABF}++:
  fast and robust angle based flattening}, ACM Transactions on Graphics, 24
  (2005), pp.~311--330.

\bibitem{sheffer2006mesh}
{\sc A.~Sheffer, E.~Praun, and K.~Rose}, {\em Mesh parameterization methods and
  their applications}, Foundations and Trends{\textregistered} in Computer
  Graphics and Vision, 2 (2006), pp.~105--171.

\bibitem{wang2007brain}
{\sc Y.~Wang, L.~M. Lui, X.~Gu, K.~M. Hayashi, T.~F. Chan, A.~W. Toga, P.~M.
  Thompson, and S.-T. Yau}, {\em Brain surface conformal parameterization using
  {R}iemann surface structure}, IEEE Transactions on Medical Imaging, 26
  (2007), pp.~853--865.

\bibitem{weber2012computing}
{\sc O.~Weber, A.~Myles, and D.~Zorin}, {\em Computing extremal quasiconformal
  maps}, Computer Graphics Forum, 31 (2012), pp.~1679--1689.

\bibitem{wong2014computation}
{\sc T.~W. Wong and H.-k. Zhao}, {\em Computation of quasi-conformal surface
  maps using discrete {B}eltrami flow}, SIAM Journal on Imaging Sciences, 7
  (2014), pp.~2675--2699.

\bibitem{wong2015computing}
{\sc T.~W. Wong and H.-K. Zhao}, {\em Computing surface uniformization using
  discrete {B}eltrami flow}, SIAM Journal on Scientific Computing, 37 (2015),
  pp.~A1342--A1364.

\bibitem{yang2009generalized}
{\sc Y.-L. Yang, R.~Guo, F.~Luo, S.-M. Hu, and X.~Gu}, {\em Generalized
  discrete {R}icci flow}, Computer Graphics Forum, 28 (2009), pp.~2005--2014.

\bibitem{yap2002grid}
{\sc P.~Yap}, {\em Grid-based path-finding}, in Conference of the Canadian
  Society for Computational Studies of Intelligence, Springer, 2002,
  pp.~44--55.

\bibitem{yueh2017efficient}
{\sc M.-H. Yueh, W.-W. Lin, C.-T. Wu, and S.-T. Yau}, {\em An efficient energy
  minimization for conformal parameterizations}, Journal of Scientific
  Computing, 73 (2017), pp.~203--227.

\bibitem{yueh2019novel}
{\sc M.-H. Yueh, W.-W. Lin, C.-T. Wu, and S.-T. Yau}, {\em A novel stretch
  energy minimization algorithm for equiareal parameterizations}, Journal of
  Scientific Computing, 78 (2019), pp.~1353--1386.

\bibitem{yung2018efficient}
{\sc C.~P. Yung, G.~P.~T. Choi, K.~Chen, and L.~M. Lui}, {\em Efficient
  feature-based image registration by mapping sparsified surfaces}, Journal of
  Visual Communication and Image Representation, 55 (2018), pp.~561--571.

\bibitem{zayer2007linear}
{\sc R.~Zayer, B.~L{\'e}vy, and H.-P. Seidel}, {\em Linear angle based
  parameterization}, in Fifth Eurographics Symposium on Geometry Processing-SGP
  2007, Eurographics Association, 2007, pp.~135--141.

\bibitem{zeng2012computing}
{\sc W.~Zeng, L.~M. Lui, F.~Luo, T.~F.-C. Chan, S.-T. Yau, and D.~X. Gu}, {\em
  Computing quasiconformal maps using an auxiliary metric and discrete
  curvature flow}, Numerische Mathematik, 121 (2012), pp.~671--703.

\bibitem{zeng2009surface}
{\sc W.~Zeng, F.~Luo, S.-T. Yau, and X.~D. Gu}, {\em Surface quasi-conformal
  mapping by solving {B}eltrami equations}, in IMA International Conference on
  Mathematics of Surfaces, Springer, 2009, pp.~391--408.

\bibitem{zeng2009generalized}
{\sc W.~Zeng, X.~Yin, M.~Zhang, F.~Luo, and X.~Gu}, {\em Generalized {K}oebe's
  method for conformal mapping multiply connected domains}, in 2009 SIAM/ACM
  Joint Conference on Geometric and Physical Modeling, 2009, pp.~89--100.

\bibitem{zhang2014unified}
{\sc M.~Zhang, R.~Guo, W.~Zeng, F.~Luo, S.-T. Yau, and X.~Gu}, {\em The unified
  discrete surface {R}icci flow}, Graphical Models, 76 (2014), pp.~321--339.

\bibitem{zhao2013area}
{\sc X.~Zhao, Z.~Su, X.~D. Gu, A.~Kaufman, J.~Sun, J.~Gao, and F.~Luo}, {\em
  Area-preservation mapping using optimal mass transport}, IEEE Transactions on
  Visualization and Computer Graphics, 19 (2013), pp.~2838--2847.

\bibitem{zou2011authalic}
{\sc G.~Zou, J.~Hu, X.~Gu, and J.~Hua}, {\em Authalic parameterization of
  general surfaces using lie advection}, IEEE Transactions on Visualization and
  Computer Graphics, 17 (2011), pp.~2005--2014.

\end{thebibliography}

\end{document}